\newcommand{\lan}{\langle}
\newcommand{\ran}{\rangle}

\newcommand{\eps}{{\varepsilon}}  
\newcommand{\veps}{{\varepsilon}}

\newcommand{\NN}{{\mathbb{N}}}

\newcommand{\QQ}{{\mathbb{Q}}}
\newcommand{\RR}{{\mathbb{R}}}
\newcommand{\EE}{{\mathbb{E}}}
\newcommand{\PP}{{\mathbb{P}}}

\newcommand{\XX}{{\mathbb{X}}}
\newcommand{\YY}{{\mathbb{Y}}}
\newcommand{\II}{{\mathbb{I}}}
\newcommand{\GG}{{\mathbb{G}}}

\newcommand{\clb}{{\mathcal{B}}}
\newcommand{\clc}{{\mathcal{C}}}

\newcommand{\cli}{{\mathcal{I}}}
\newcommand{\clp}{{\mathcal{P}}}
\newcommand{\cld}{{\mathcal{D}}}
\newcommand{\cle}{{\mathcal{E}}}
\newcommand{\clf}{{\mathcal{F}}}
\newcommand{\Gmc}{{\mathcal{G}}}

\newcommand{\Lmc}{{\mathcal{L}}}

\newcommand{\clm}{{\mathcal{M}}}

\newcommand{\clu}{{\mathcal{U}}}
\newcommand{\clv}{{\mathcal{V}}}

\newcommand{\Ybf}{{\mathbf{Y}}}

\newcommand{\Psib}{\mathbf{\Psi}}
\newcommand{\zero}{{\boldsymbol{0}}}
\newcommand{\one}{{\boldsymbol{1}}}

\newcommand{\bbd}{{\boldsymbol{b}}}\newcommand{\Bbd}{{\boldsymbol{B}}}

\newcommand{\ebd}{{\boldsymbol{e}}}

\newcommand{\gbd}{{\boldsymbol{g}}}

\newcommand{\Lbd}{{\boldsymbol{L}}}
\newcommand{\Mbd}{{\boldsymbol{M}}}

\newcommand{\pbd}{{\boldsymbol{p}}}

\newcommand{\ubd}{{\boldsymbol{u}}}
\newcommand{\vbd}{{\boldsymbol{v}}}

\newcommand{\xbd}{{\boldsymbol{x}}}\newcommand{\Xbd}{{\boldsymbol{X}}}
\newcommand{\ybd}{{\boldsymbol{y}}}\newcommand{\Ybd}{{\boldsymbol{Y}}}
\newcommand{\zbd}{{\boldsymbol{z}}}\newcommand{\Zbd}{{\boldsymbol{Z}}}

\newcommand{\Om}{\Omega}
\newcommand{\om}{\omega}

\newcommand{\la}{\lambda}

\newcommand{\lest}{\le_{\mbox{{\tiny st}}}}

\documentclass[10pt]{article}
\usepackage[portrait,margin=2.54cm]{geometry}

\usepackage{amsfonts}

\usepackage{amssymb}
\usepackage{mathrsfs}
\usepackage{soul}
\usepackage{hyperref}
\usepackage{amssymb,amsthm,amsmath,amsfonts,amsbsy,latexsym}
\usepackage{graphicx}
\usepackage{upref,setspace}
\usepackage{enumerate}
\usepackage{paralist}
\usepackage{color}

\usepackage{mathtools}

\definecolor{expcol}{rgb}{1.0,0.5,0.5}
\definecolor{ecol}{rgb}{0.0, 0.0, 1.0}

\numberwithin{equation}{section}
\numberwithin{figure}{section}
\numberwithin{table}{section}
\newtheorem{lemma}{Lemma}[section]
\newtheorem{proposition}[lemma]{Proposition}

\newtheorem{theorem}[lemma]{Theorem}

\newtheorem{remark}[lemma]{Remark}

\newtheorem{corollary}[lemma]{Corollary}
\newtheorem{definition}[lemma]{Definition}
\numberwithin{equation}{section}

\mathtoolsset{showonlyrefs}
\newcommand{\SB}{\color{black}}

\begin{document}

\title{Extremal Invariant Distributions of Infinite Brownian Particle Systems with Rank Dependent Drifts}
\author{Sayan Banerjee, Amarjit Budhiraja}
\maketitle

\begin{abstract}
\noindent Consider an infinite collection of particles on the real line moving according to independent Brownian motions and such that the $i$-th particle from the left gets the drift $g_{i-1}$. The case where $g_0=1$ and $g_{i}=0$ for all $i \in \NN$ corresponds to the well studied infinite Atlas model. Under conditions on the drift vector $\gbd = (g_0, g_1, \ldots)'$ it is known that the Markov process corresponding to the gap sequence of the associated ranked particles has a continuum of product form stationary distributions
$\{\pi_a^{\gbd}, a \in S^{\gbd}\}$ where $S^{\gbd}$ is a semi-infinite interval of the real line. In this work we show that all of these stationary distributions are extremal and ergodic. We also prove 
that any product form stationary distribution of this Markov process that satisfies a mild integrability condition must be $\pi_a^{\gbd}$ for some $a \in S^{\gbd}$. These results are new even for the infinite Atlas model. The work makes progress on the open problem of characterizing all the invariant distributions of general competing Brownian particle systems interacting through their relative ranks. Proofs rely on synchronous and mirror coupling of Brownian particles and properties of the intersection local times of the various particles in the infinite system.

\noindent\newline


\end{abstract}

\section{Introduction}

\label{sec:intro}  

\subsection{Background}Consider  a  collection (finite or infinite) of particles on the real line moving according to mutually independent Brownian motions and such that the $i$-th particle from the left gets a constant drift $g_{i-1}$. The special case when $g_0=1$ and $g_i =0$ for $i\in \NN$ is the well studied Atlas model. We  refer to the general setting as the $\gbd$-Atlas model, where $\gbd = (g_0, g_1, \ldots)'$. Such particle systems were originally introduced in stochastic portfolio theory \cite{fernholz2002stochastic,banner2005atlas,fernholz2009stochastic} as models for stock growth evolution in equity markets and have been investigated extensively in recent years in several different directions. In particular, characterizations of such particle systems as uniform  scaling limits of jump processes with local interactions on integer lattices, such as the totally asymmetric simple exclusion process, have been studied in \cite{karatzas2016systems}. Various types of results for  the asymptotic behavior of the empirical measure of the particle states have been studied, such as propagation of chaos, characterization of the associated McKean-Vlasov equation and nonlinear Markov processes 
\cite{shkol, jourdain2008propagation}, large deviation principles \cite{dembo2016large}, characterizing the asymptotic density profile and the  trajectory of the leftmost particle via Stefan free-boundary problems \cite{cabezas2019brownian}.  These particle systems also have  close connections with Aldous' ``Up the river'' stochastic control problem \cite{aldous41up}, recently solved in \cite{tang2018optimal}.
Results on wellposedness of the associated stochastic differential equations (in the weak and strong sense) and on absence of triple collisions (three particles at the same place at the same time) have been studied in \cite{bass1987uniqueness,shkol2011levy, ichiba2013strong, ichiba2010collisions,sarantsev2015triple,ichiba2017yet}.

One important direction of investigation has been in the study of the long-time behavior of such particle systems. For finite particle systems, under conditions on the drift vector $\gbd$, it follows from results of Harrison and Williams \cite{harrison1987multidimensional,harrison1987brownian} that the multidimensional reflected Brownian motion describing the evolution of the gaps between the ranked particles has a unique stationary (invariant) distribution (see \cite{PP}). It is known that convergence of the law at time $t$ to this stationary distribution, as $t \to \infty$, occurs at a geometric rate \cite{budlee}. Rates of convergence to stationarity, depending explicitly on  the drift vector $\gbd$ and dimension have been obtained in \cite{ichiba2013convergence,BanBudh,banerjee2020dimension}.

In the current work we are interested in infinite particle systems. One basic result on long-time behavior of such particle systems was obtained in \cite{PP} which showed that for the infinite Atlas model, i.e. when $\gbd = \gbd^1=(1,0,0, \ldots)'$, the process describing the gaps between the ranked particles in the system has a simple product form stationary distribution given as $\pi_0 = \otimes_{n=1}^{\infty} \mbox{Exp}(2)$ {\SB(here and later, for $a>0$, $\mbox{Exp}(a)$ denotes the Exponential distribution with mean $1/a$)}. The paper \cite{PP} also conjectured that this is the unique stationary distribution of the (gap sequence in the) infinite Atlas model.
However, this was shown to be false in  \cite{sarantsev2017stationary} who gave an uncountable collection of product form stationary distributions for the gaps in the infinite Atlas model defined as
\begin{equation}\label{altdef}
\pi_a \doteq \bigotimes_{i=1}^{\infty} \mbox{Exp}(2 + ia), \ a \ge 0.
\end{equation}
As in the finite dimensional settings it is of interest to investigate convergence of the laws at time $t$ to the stationary distributions as $t \to \infty$. Due to the multiplicity of stationary distributions, a meaningful goal is to understand the local stability structure of this infinite dimensional stochastic dynamical system and identify the basins of attraction of the various stationary distributions.
Such results, describing (weak and strong) domains of attraction of $\pi_0$ have been obtained in \cite{AS,DJO,banerjee2022domains} and (weak) domains of attraction of $\pi_a$ ($a\neq 0$) in
\cite{banerjee2022domains}. 
{\SB Results analogous to \cite{sarantsev2017stationary, AS} for
 two-sided infinite Brownian systems have been obtained  in \cite{sarantsev2017two}.}

\subsection{Goals and results}
Although the above results give us a good understanding of the local stability structure of the infinite Atlas model, the picture that one has is far from complete. A key obstacle here is that a full characterization of all extremal invariant distributions of the infinite Atlas model is currently an open problem. The goal of this work is to make some progress towards this goal  and, moreover, provide some characterization of the structure of the set of invariant distributions. We will in fact consider the more general setting of the $\gbd$-Atlas model where the drift vector $\gbd \in \cld$, {\SB with}
\begin{equation}\label{eq:clddefn}
	\cld \doteq \Big\{\gbd = (g_0, g_1, \ldots)' \in \RR^{\infty}:  \sum_{i=0}^{\infty} g_i^2 <\infty \Big\}.
\end{equation}
For this setting it is known from the work of \cite{sarantsev2017stationary} that, once more, the process  associated with the gap sequence of the ranked particle system has a continuum of stationary distributions given as
\begin{equation}\label{eq:piagbddefn}
	\pi_a^{\gbd} \doteq \otimes_{n=1}^{\infty} \mbox{Exp}(n (2\bar g_n+a)), \;\; a > -2 \inf_{n \in \NN} \bar g_n,\end{equation}
where $\bar g_n \doteq \frac{1}{n}(g_0 + \cdots + g_{n-1})$.  {\SB In the special case where $\gbd \in \cld_1$, {\SB with}
\begin{equation}\label{cldzdefn}
	\cld_1 \doteq \{\gbd \in \cld: \text{ there exist } N_1 < N_2 < \dots \rightarrow \infty \text{ such that } \bar g_k > \bar g_{N_j}, \ k=1,\dots,N_j-1, \text{ for all } j \ge 1\},\end{equation}
 $\pi_a^{\gbd}$ is also an invariant distribution for $a=-2 \inf_{n \in \NN} \bar g_n  = -2\lim_{j \rightarrow \infty} \bar g_{N_j}$ (see \cite[Section 4.2]{AS}). 
Note that $\gbd^1 \in \cld_1$ whereas the zero drift $\gbd = (0, 0, \ldots)'$ is in $\cld$ but not in $\cld_1$.
 Roughly speaking, a drift $\gbd$ lying in $\cld_1$ produces a `stabilizing interaction' in the subsystem of the lowest $N_j$ particles for any $j \ge 1$, 
 due to which the gaps between them stabilize in time owing to stronger upward average drifts of the lower particles in this subsystem in comparison to the average drift of all the $N_j$ particles. This intrinsic stabilizing influence from the drift of the particles leads to an additional stationary distribution, namely  $\pi_{ -2\lim_{j \rightarrow \infty} \bar g_{N_j}}$ (in comparison to a $\gbd \in \cld\setminus \cld_1 $).
For $\gbd \in \cld\setminus \cld_1$, where such a  mechanism  is absent, local stability essentially arises from configurations with a
 rapidly increasing density of particles as one moves away from the lowest one and hence one only obtains ``dense stationary distributions" corresponding to $a  > -2 \inf_{n \in \NN} \bar g_n$.} 
 
 Using {\SB Kakutani's theorem \cite{kakutani1948equivalence}} it is easy to verify  that, for different values of $a$, the probability measures 
$\pi_a^{\gbd}$ are mutually singular. 
These distributions are also special in that they have a product form structure. In particular, if the initial distribution of the gap process is chosen to be one of these distributions, then the laws of distinct gaps at any fixed time are independent despite these gaps having a highly correlated temporal evolution mechanism (see \eqref{order_SDE}-\eqref{gapdef}). {\SB We now describe the two main results of this work.}

\textbf{First result: Extremality.} In Theorem \ref{thm_main} we show that for each $\gbd \in \cld$ and $a> -2 \inf_{n \in \NN} \bar g_n$, $\pi_a^{\gbd}$ is an extremal invariant distribution for the gap sequence process of the $\gbd$-Atlas model.
 Further, if $\gbd \in \cld_1$,
 $\pi_a^{\gbd}$ is also an extremal invariant distribution for $a=-2 \inf_{n \in \NN} \bar g_n  = -2\lim_{j \rightarrow \infty} \bar g_{N_j}$; in particular for the infinite Atlas model $\pi_a = \pi_a^{\gbd^1}$ is extremal for all $a\ge 0$.
 From equivalence between extremality and ergodicity (cf. Lemma \ref{lemerg}) it then follows that all these invariant distributions are ergodic as well. {\SB This result also identifies non-trivial subsets in the weak domain of attraction of $\pi_a^{\gbd}$ for each $a> -2 \inf_{n \in \NN} \bar g_n$ ($a \ge -2 \inf_{n \in \NN} \bar g_n$ if $\gbd \in \cld_1$); see Corollary \ref{doacor}.}

 Questions about extremality and ergodicity of stationary distributions have been addressed previously in the context of interacting particle systems on countably infinite graphs (see \cite{liggett1976coupling,andjel1982invariant,sethuraman2001extremal,balazs2007existence} and references therein). However, in all these cases, the interactions are Poissonian, {\SB namely, the dynamics is given in terms of jumps of particles to neighboring vertices in a countably infinite graph at epochs of Poisson processes associated with edges or vertices.} This enables one to use the (explicit) generator of the {\SB associated continuous-time jump-Markov processes} in an effective way. The interactions in rank based diffusions are very `singular' owing to the local time based dynamics (see \eqref{order_SDE}) and generator based methods seem to be less tractable. Furthermore, unlike previous works,  the state space  for the gap process (i.e. $\mathbb{R}_+^{\infty}$) is not countable and has a non-smooth boundary, and the process has intricate interactions (oblique reflections) with the boundary. Hence, proving extremality requires new techniques. Our proofs are based on  constructing appropriate couplings for these infinite dimensional diffusions which then allow us to prove suitable invariance properties (see e.g. \eqref{eq:gamone}) and a certain `directional strong Feller property' (see Proposition \ref{prop:prop2}). {\SB Such coupling techniques, based on `mirror' couplings of driving Brownian motions, are novel in the context of infinite rank based diffusions and provide a new method for establishing semigroup continuity properties for such processes. Moreover, the coupling approach introduced in the paper
 has the potential to be applicable to broader families of infinite-dimensional diffusion processes for which such directional strong Feller properties may be useful, e.g. in analysis of the  ergodic behavior.} Although our setting and methods are very different, at a high level, the approach we take, {\SB of proving extremality by showing the a.e. constancy of suitable invariant functions}, is inspired by the papers \cite{sethuraman2001extremal,balazs2007existence}.

{\bf Second result: Characterization of product form stationary distributions.}  
A natural question is whether there are any other product form stationary distributions of the $\gbd$-Atlas model than the ones identified in \cite{sarantsev2017stationary}. Our second main result (Theorem \ref{thm_main2}) answers this question in the negative under certain conditions by showing that if $\gbd \in \cld_1$
and $\pi$ is a product form stationary distribution of the $\gbd$-Atlas model satisfying a mild integrability condition (see \eqref{eq:satinteg}) then it must be $\pi_a^{\gbd}$ for some 
$a \ge -2\lim_{j \rightarrow \infty} \bar g_{N_j}$. {\SB Furthermore, this result gives a novel probabilistic interpretation to $a$   in terms of the resulting force acting on a tagged particle under the combined influence of hardcore interactions (collision local times) and soft potentials (drift terms). See Remark \ref{arem} for this interpretation and for a conjecture  that is suggested by this interpretation}.

\subsection{Proof ideas}
We now make some comments on proofs. The key step in proving the extremality of $\pi_a^{\gbd}$ is to establish that any bounded measurable function $\psi$ on $\RR_+^{\infty}$ that is $\pi_a^{\gbd}$-a.e. invariant, under the action of the semigroup of the Markov process corresponding to the $\gbd$-Atlas gap sequence, is constant $\pi_a^{\gbd}$-a.e.
If $\gbd = \gbd^1$ and $a=0$, we have that $\pi_a^{\gbd} = \otimes_{i=1}^{\infty} \mbox{Exp}(2)$, and therefore the coordinate sequence $\{Z_i\}_{i=1}^{\infty}$ is iid under $\pi_a^{\gbd}$. In this case, from the Hewitt-Savage zero-one law it suffices to show that $\psi$ is $\pi_a^{\gbd}$-a.e. invariant under all finite permutations of the coordinates of $\RR_+^{\infty}$. For this, in turn, it suffices to simply prove the above invariance property for transpositions of the $i$-th and $(i+1)$-th coordinates, for all $i \in \NN$. For a general $\pi_a^{\gbd}$, the situation is  more involved as the coordinate sequence $\{Z_i\}_{i=1}^{\infty}$  is not iid any more. Nevertheless, from the scaling properties of Exponential distributions it follows that, with $c_n \doteq 2[n (2\bar g_n+a)]^{-1}$, the sequence 
$\{\tilde Z_n\}_{n\ge 1}$, defined as $\tilde Z_n = c_n^{-1} Z_n$, $n \in \NN$, is iid under $\pi_a^{\gbd}$. In this case, in order to invoke the Hewitt-Savage zero-one law, one needs to argue that for each $i$ the map $\psi$ is $\pi_a^{\gbd}$-a.e. invariant under the transformation that takes the $(i, i+1)$ coordinates $(z_i, z_{i+1})$ to $(\frac{c_i}{c_{i+1}} z_{i+1}, \frac{c_{i+1}}{c_{i}} z_{i})$ and keeps the remaining coordinates the same. Establishing this property is at the heart of the proof of Theorem \ref{thm_main}. A key technical idea in the proof is the construction of a mirror coupling of the first $i+1$ Brownian motions, and synchronous coupling of the remaining Brownian motions,
in the evolution of the ranked infinite $\gbd$-Atlas model corresponding to a pair of nearby initial configurations. Estimates on the probability of a successful coupling, before any of the first $i$-gap processes have hit zero or the lowest $i$-particles have interacted with the higher ranked particles (in a suitable sense), are some of the important ingredients in the proof. {\SB We refer the reader to Section \ref{sec:pfov1} for additional comments on the proof strategy.}

The proof of Theorem \ref{thm_main2} hinges on establishing a key identity for expectations, under the given product form invariant measure $\pi$, of certain integrals involving the state process of the $i$-th gap and the collision local time for the $(j-1)$-th and $j$-th particle, for $i\neq j$ (see Lemma \ref{lem:lociden}). This identity is a consequence of the product form structure of $\pi$ and basic results on local times of continuous semimartingales. 
One subtlety here is that although the product form structure of an invariant measure $\pi$ implies that the laws of the various gaps at any fixed time $t$ (when the process is initiated at $\pi$) are independent, the laws of the paths of the various gap processes are not, and thus one cannot immediately deduce the independence beween the state of the $i$-th gap at time $t$ and the local time of the $j$-th gap at $0$, at time $t$, for $i\neq j$.
By using the form of the dynamics of the $\gbd$-Atlas model, the identity in Lemma \ref{lem:lociden} allows us to 
{\SB obtain a recursive system of equations for the moment generating functions of the coordinate projections of $\pi$  that can then  be solved explicitly
from which it is then readily seen that $\pi$ must be $\pi_a^{\gbd}$ for a suitable value of $a$. See Section \ref{sec:pfov2} for additional comments on the proof idea.}

\subsection{Open problem}
Of course it is immediate to construct non-product form stationary distributions of the $\gbd$-Atlas model by considering mixtures of the above product stationary distributions, however one can ask if these mixtures are all the invariant measures of the $\gbd$-Atlas model. For the cases where $\gbd = (0,0, \ldots)'$ this question was answered in the affirmative in \cite[Theorem 4.2]{RuzAiz} under certain  integrability constraints on the denseness of particle configurations. For a general $\gbd$ providing such a complete characterization is a challenging open problem.

In the context of interacting particle systems on countably infinite graphs, the analogous problem has been solved completely in a few cases such as the simple exclusion process \cite{liggett1976coupling} and the zero range process \cite{andjel1982invariant} where the extremal probability measures are fully characterized as an explicit collection of certain product form measures.
However, in these models the particle density associated with distinct extremal measures are scalar multiples of each other owing to certain homogeneity properties in the dynamics (see, for example, \cite[Theorem 1.10]{andjel1982invariant}). This, along with the Poissonian nature of the interactions, enables one to prove useful monotonicity properties of the `synchronously coupled' dynamics (see \cite[Section 2]{liggett1976coupling} and \cite[Section 4]{andjel1982invariant}) using generator methods that are crucial to the above characterization results. 

A key challenge in extending these methods to rank based diffusions of the form considered here is that, in addition to the singular local time interactions, the point process associated with the configuration of particles with gaps distributed as $\pi_a^{\gbd}$ has an intensity function  $\rho(x)$, that grows exponentially as $x\to \infty$ when $a>0$ and, due to a nonlinear dependence on $a$, lacks the scalar multiple property for distinct values of $a$. This is a direct consequence of the inhomogeneity of the topological interactions in our particle systems where the local stability in a certain region of the particle cloud is affected both by the density of particles in the neighborhood and their relative ranks in the full system.  
Moreover, unlike the above interacting particle systems,  in rank based diffusions, when the initial gaps are  given by a stationary distribution,  the point process of particles is typically not stationary. This phenomenon, where  the gaps are stationary while the associated point process  is not, referred to as \emph{quasi-stationarity} in \cite{RuzAiz}, is technically challenging. We note that this latter paper studies one setting where the intensity function grows exponentially and the particle density lacks the scalar multiple property for distinct values of $a$. However their setting, in the context of rank based diffusions, corresponds to the
case  $\gbd = (0,0, \ldots)'$, where the unordered particles behave like independent standard Brownian motions, and this fact is crucially exploited in \cite{RuzAiz}.

\subsection{Organization}Rest of the paper is organized as follows. We close this section by summarizing the main notation used in this work. In Section \ref{sec:mainres} we give the precise formulation of the model. In Section \ref{sec:mr}, we describe the questions of interest and give our main results. Finally Sections \ref{sec:prffirstthm} and \ref{sec:mainsecthm} give the proofs of our main results, namely Theorems \ref{thm_main} and \ref{thm_main2}, respectively.

\subsection{Notation}

The following notation will be used.  
Let {\SB $\NN_0 \doteq \{0,1,2,\dots\}$}, $\RR^{\infty} \doteq \{(x_0, x_1, \ldots)': x_i \in \RR,\, i \in \NN_0\}$ and $\RR_+^{\infty} \doteq \{(x_0, x_1, \ldots)': x_i \in \RR_+,\,  i \in \NN_0\}$.
We will equip $\RR^{\infty}_+$ with the partial order `$\le$' under which $\xbd \le \ybd$ for $\xbd = (x_i)_{i \in \NN_0}$, $\ybd = (y_i)_{i \in \NN_0}$, if $x_i \le y_i$ for all $i \in \NN_0$.  Borel $\sigma$-fields on a metric space $S$ will be denoted as $\clb(S)$ and the space of probability measures on $(S, \clb(S))$ will be denoted as $\clp(S)$ which is equipped with the topology of weak convergence.
We will denote by $\QQ$ the set of rationals.
For a Polish space $S$ with a partial order `$\le$', we say for $\gamma_1, \gamma_2 \in \clp(S)$, $\gamma_1 \lest \gamma_2$, if there are $S$-valued random variables $\Xbd_1, \Xbd_2$ given on a common probability space with $\Xbd_i$ distributed as $\gamma_i$, $i = 1,2$, and $\Xbd_1 \le \Xbd_2$ a.s. {\SB Let $\XX \doteq  \clc([0,\infty): \RR_+^{\infty})$ which is equipped with the topology of local uniform convergence (with $\RR_+^{\infty}$ equipped with the product topology). For $\gamma \in \clp(\RR_+^{\infty} )$,
 let $L^2(\gamma)$ be the collection of  all measurable $\psi: \RR_+^{\infty} \to \RR$ such that
$\int_{\RR_+^{\infty}} |\psi(\zbd)|^2  \gamma(d\zbd) <\infty$.
We denote the inner-product and the norm on $L^2(\gamma)$ as $\lan \cdot, \cdot \ran$ and $\|\cdot\|$ respectively.}
\section{Model Formulation}
\label{sec:mainres}
Let
\begin{equation*}
	\clu \doteq \Big\{\xbd = (x_0, x_1, \ldots)' \in \RR^{\infty}: \sum_{i=0}^{\infty} e^{-\alpha [(x_i)_+]^2} < \infty \mbox{ for all } \alpha >0\Big\}.
\end{equation*}
Recall $\cld$ defined in \eqref{eq:clddefn}.
Following \cite{AS}, an $\xbd \in \RR^{\infty}$ is called {\em rankable} if there is a one-to-one map $\pbd$ from $\NN_0$ onto itself such that $x_{\pbd(i)} \le x_{\pbd(j)}$ whenever $i<j$, $i,j \in \NN_0$. It is easily seen that any $\xbd \in \clu$ is locally finite and hence rankable. For an $\xbd$ that is rankable we denote the unique permutation map $\pbd$ as above which breaks ties in the lexicographic order by $\pbd_{\xbd}$.
For a sequence $\{W_i\}_{i\in \NN_0}$ of mutually independent standard Brownian motions given on some filtered probability space, and $\ybd = (y_0, y_1, \ldots)' \in \clu$, $\gbd = (g_0, g_1, \ldots)' \in \cld$, consider the following system of equations.
\begin{equation}\label{eq:unrank}
	dY_i(t) = \left[\sum_{k=0}^{\infty} \one(\pbd_{\Ybf(t)}(k) = i) g_k\right] dt + dW_i(t), 
\end{equation}
$Y_i(0) = y_i$,  $i \in \NN_0$,
where for $t\ge 0$,
$\Ybf(t) = (Y_0(t), Y_1(t), \ldots)'$.

The following result  is from \cite{AS} (see Theorems 3.2 and Lemma 3.4 therein).
\begin{theorem}[\cite{AS}]\label{sarthm}
For every $\gbd \in \cld$ and $\ybd \in \clu$ there is a unique weak solution $\Ybd(\cdot)$ to \eqref{eq:unrank}. 	Furthermore
$P(\Ybf(t)\in \clu \mbox{ for all } t \ge 0)=1$ and for any $T<\infty$ and $m \in (0, \infty)$, the set $\{i \in \NN_0: \inf_{t \in [0,T]} Y_i(t) \le m\}$ is finite a.s.
\end{theorem}
When $\gbd = \gbd^1 = (1, 0,0, \ldots)'$ the process given by the above theorem is the well known (standard) infinite Atlas model. In general, the unique in law solution process given by Theorem \ref{sarthm} will be referred to as the $(\gbd, \ybd)$-infinite Atlas model.
Since this solution process $\Ybf(t)$, under the conditions of the above theorem, is rankable a.s., we can define the ranked process $\{Y_{(i)}(t), i \in \NN_0\}_{t\ge 0}$ 
that gives the unique ranking of $\Ybf(t)$ (in which ties are broken in the lexicographic order) such that $Y_{(0)}(t) \le Y_{(1)}(t) \le Y_{(2)}(t) \le \cdots$.
From \cite[Lemma 3.5]{AS}, the processes defined by
\begin{equation}\label{eq:bistar}
B^*_i(t) \doteq \sum_{j=0}^{\infty}\int_0^t\mathbf{1}(Y_j(s) = Y_{(i)}(s)) dW_j(s), \ i \in \mathbb{N}_0, t \ge 0,
\end{equation}
are independent standard Brownian motions which can be used to write down the following stochastic differential equation (SDE) for $\{Y_{(i)}(\cdot) : i \in \mathbb{N}_0\}$:
\begin{equation}\label{order_SDE}
dY_{(i)}(t) = g_i dt + dB^*_i(t) - \frac{1}{2}dL^*_{i+1}(t) + \frac{1}{2} dL^*_{i}(t), \ t \ge 0, \;\;   Y_{(i)}(0) = y_{(i)},\, i \in \mathbb{N}_0.
\end{equation}
Here, $L^*_{0}(\cdot) \equiv 0$ and for $i \in \mathbb{N}$, $L^*_i(\cdot)$ denotes the local time of collision between the $(i-1)$-th and $i$-th particles, that is, the unique non-decreasing continuous process satisfying $L^*_i(0)=0$ and $L^*_i(t) = \int_0^t\mathbf{1}\left(Y_{(i-1)}(s)=Y_{(i)}(s)\right)dL^*_i(s)$ for all $t \ge 0$.
The gap process for the  $(\gbd, \ybd)$-infinite Atlas model is the $\RR_+^{\infty}$-valued process $\Zbd(\cdot)=$ $(Z_1(\cdot), Z_2(\cdot), \dots)'$ defined by
\begin{equation}\label{gapdef}
Z_i(\cdot) \doteq Y_{(i)}(\cdot) - Y_{(i-1)}(\cdot), \ i \in \mathbb{N}.
\end{equation}
Let
\begin{equation}\label{clvdef}
\clv \doteq \{\zbd \in \RR_+^{\infty}: \mbox{ for some } \ybd \in \clu, \zbd = (y_{(1)}-y_{(0)}, y_{(2)}-y_{(1)}, \ldots)'\}.
\end{equation}
Note that if $\zbd \in \clv$ then $\ybd(\zbd) \doteq (0, z_1, z_1+z_2, \ldots)' \in \clu$ and if $\ybd \in \clu$ then
$\zbd(\ybd) \doteq (y_{(1)}-y_{(0)}, y_{(2)}-y_{(2)}, \ldots)' \in \clv$.
Thus given $\zbd \in \clv$ we can define a unique in law stochastic process $\{\Zbd(t)\}_{t\ge 0}$ with values in $\clv$ that can be viewed as a $\clv$-valued Markov process  referred to as the {\em gap process of the $\gbd$-Atlas model}. The Markov property of  $\Zbd(\cdot)$ needs justification which we have sketched for completeness in Remark \ref{MP} at the end of the section.

The following result identifies an important family of stationary distributions of this Markov process. The first statement in the theorem is from \cite[Theorem 1.6]{sarantsev2017stationary}. The second statement is due to \cite[Section 4.2]{AS} and \cite[Remark 4]{sarantsev2017stationary}. Recall $\cld_1$ defined in \eqref{cldzdefn}.
\begin{theorem}[\cite{sarantsev2017stationary,AS}]\label{thm:406n}
	Let $\gbd \in \cld$. Define for $n \in \NN$, $\bar g_n \doteq \frac{1}{n}(g_0 + \cdots + g_{n-1})$.
	Then for each $a > -2 \inf_{n \in \NN} \bar g_n$, the probability measure $\pi_a^{\gbd}$ on $\RR_+^{\infty}$ defined as 
	$$\pi_a^{\gbd} \doteq \otimes_{n=1}^{\infty} \operatorname{Exp}(n (2\bar g_n+a))$$
	is a stationary distribution for the  gap process of the $\gbd$-Atlas model.
	 Furthermore, if $\gbd \in \cld_1$, the above statement holds also for the case $a = -2 \inf_{n \in \NN} \bar g_n =-2\lim_{j \rightarrow \infty} \bar g_{N_j}$.
	In particular, when $\gbd = \gbd^1 \doteq (1, 0, \ldots)'$ (infinite Atlas model), 
	$\pi_a^{\gbd^1}\doteq \otimes_{n=1}^{\infty} \operatorname{Exp}(2 + na)$ is a stationary distribution for the  gap process for all $a \ge 0$.
\end{theorem}

The existence of the limit $\lim_{j \rightarrow \infty} \bar g_{N_j}$ and the equality $\inf_{n \in \NN} \bar g_n =\lim_{j \rightarrow \infty} \bar g_{N_j}$ follow from the definition of $\cld_1$ (see first paragraph of Section \ref{sec:mainsecthm}).

As an immediate consequence of Kakutani's theorem \cite{kakutani1948equivalence} we see that for any $\gbd \in \cld$ and $a, a' > -2 \inf_{n \in \NN} \bar g_n$ (and $a, a' \ge -2\inf_{n \in \NN} \bar g_n$ when $\gbd \in \cld_1$), $a \neq a'$, the measures
$\pi_a^{\gbd}$ and $\pi_{a'}^{\gbd}$ are mutually singular.

\begin{remark}[Markov property of the gap process]\label{MP}
The ranked process $\{Y_{(i)}(\cdot)\}_{i \in \NN_0}$ is formally constructed in \cite{AS} as an `approximative version' using limits of finite-dimensional reflected diffusions. Namely, for any fixed $k \in \mathbb{N}_0$, $\{Y_{(i)}(\cdot)\}_{0 \le i \le k}$ is obtained as an almost sure limit as $m \rightarrow \infty$, uniformly over compact time intervals, of the first $k+1$ coordinates of the reflected diffusion $\{Y_{(i)}^{m}(\cdot)\}_{0 \le i \le m}$ given by  the SDE \eqref{order_SDE} with $Y_{(i)}^m(0) = y_{(i)}$, $L^*_{m+1}(\cdot) \equiv 0$, and a given collection $\{B_i^*\}_{i\in \NN_0}$ of iid standard Brownian motions (see also Definition \ref{def:sttappver}). Fix any time $t \ge 0$. Define for any $m \in \mathbb{N}_0$ the process $\{Y^{m,t}_{(i)}(t + \cdot)\}_{0 \le i \le m}$ similarly by setting $Y_{(i)}^{m,t}(t) = Y_{(i)}(t)$ and driven by the Brownian motions { $\{B^*_i(t+\cdot) - B^*_i(t)\}_{0 \le i \le m}$} via the SDE \eqref{order_SDE} (again with the local time for $i=m+1$ set to zero). Define $Z^{m,t}_i$, $1\le i \le m$, the gap process sequence associated with
$Y_{(i)}^{m,t}$, $0\le i \le m$.
Fix $k \in \mathbb{N}_0$, $s>0$, and any bounded continuous function 
$f: \mathbb{R}^k \rightarrow \mathbb{R}$. Let $\Gmc_{\le t} \doteq \sigma \{Z_{i}(u): u\le t, i \in \NN\}$ and $\Gmc_{t} \doteq \sigma \{Z_{i}(t): i \in \NN\}$. For any $m \ge k$ and $t>0$, as { $\{B^*_i(t+\cdot) - B^*_i(t)\}_{i \in \mathbb{N}_0}$} is independent of $\Gmc_{\le t}$,
$$
\mathbb{E}\left(f\left(Z^{m,t}_{1}(t + s),\dots, Z^{m,t}_{k}(t + s)\right)\, \Big|\, \Gmc_{\le t}\right) = \mathbb{E}\left(f\left(Z^{m,t}_{1}(t + s),\dots, Z^{m,t}_{k}(t + s)\right)\, \Big| \,\Gmc_t\right).
$$
Thus, to deduce the Markov property, it suffices to show that $\left(Z^{m,t}_{1}(t + s),\dots, Z^{m,t}_{k}(t + s)\right)$ converges almost surely to $\left(Z_{1}(t + s),\dots, Z_{k}(t + s)\right)$ as $m \rightarrow \infty$ which will follow from the a.s. convergence of
$\left(Y^{m,t}_{(0)}(t + s),\dots, Y^{m,t}_{(k)}(t + s)\right)$  to $\left(Y_{(0)}(t + s),\dots, Y_{(k)}(t + s)\right)$. The latter can be shown by exploiting monotonicity properties of rank-based diffusions \cite[Theorem 3.2 and Corollary 3.9]{AS2} as follows.

Fix $\epsilon>0$. By the construction of the process $\{Y_{(i)}(\cdot)\}_{i \in \NN_0}$ in \cite{AS}, we can find (random) $m_0 \in \mathbb{N}_0$ such that, for any 
$m \ge m_0$,  $|Y^{m}_{(i)}(t + s) - Y_{(i)}(t + s)| < \epsilon$, for $0\le i \le k$. Now fix $m \ge m_0$.
Note that $Y^{m,t}_{(i)}(t) \le Y^{m}_{(i)}(t)$ for $ 0 \le i \le m$ by \cite[Corollary 3.9]{AS2}, and hence by  \cite[Theorem 3.2]{AS2}, 
$$
{ Y^{m,t}_{(i)}(t+s) \le Y^{m}_{(i)}(t+s) \le Y_{(i)}(t+s) + \epsilon  \ \text{ for } \  0 \le i \le m.}
$$
To construct a lower bounding process for $Y^{m,t}_{(i)}(t+s)$, define for $m' \ge m$ the process $\{Y^{m',m,t}_{(i)}(t + \cdot)\}_{0 \le i \le m}$ started with $Y^{m',m,t}_{(i)}(t) = Y^{m'}_{(i)}(t)$ and driven by the Brownian motions { $\{B^*_i(t+\cdot) - B^*_i(t)\}_{0 \le i \le m}$} via the SDE \eqref{order_SDE} (and again setting the local time for $i=m+1$ to be zero). 
{ By the construction of the process $\{Y_{(i)}(\cdot)\}_{i \in \NN_0}$ in \cite{AS}, one can choose $m' = m'(m,\epsilon)$ large enough so that
$Y^{m'}_{(i)}(t) \le Y_{(i)}(t)+ \epsilon$ for $0 \le i \le m$. Moreover, if we consider the  translation of the system $\{Y^{m,t}_{(i)}(t + \cdot)\}_{0 \le i \le m}$ by $\epsilon$, then the $(m+1)$-particle process started from $\{Y_{(i)}(t)+ \epsilon : 0 \le i \le m\}$ at time $t$ and driven by the Brownian motions $\{B^*_i(t+\cdot) - B^*_i(t)\}_{0 \le i \le m}$ has particle locations at time $t+s$ given by $\{Y^{m,t}_{(i)}(t+s)+ \epsilon : 0 \le i \le m\}$. Hence, using \cite[Theorem 3.2 and Corollary 3.9]{AS2}, we have
\begin{equation}
Y_{(i)}(t+s) - \epsilon \le Y^{m'}_{(i)}(t + s) \le Y^{m',m,t}_{(i)}(t + s) \le Y^{m,t}_{(i)}(t+s) + \epsilon \ \text{ for } \  0 \le i \le m.
\end{equation}
The first inequality above holds because $m'\ge m \ge m_0$. 
We conclude from the above two displays that}
\begin{equation}
Y_{(i)}(t + s) - 2\epsilon \le Y^{m,t}_{(i)}(t+s) \le Y_{(i)}(t + s) + \epsilon \ \text{ for } \  0 \le i \le k.
\end{equation}
This gives the desired almost sure convergence as $\epsilon>0$ is arbitrary.
\end{remark}

\section{Main Results}\label{sec:mr}

We are interested in the extremality properties of the probability measures $\pi_a^{\gbd}$. We also ask whether these are the only product form stationary distributions.

We begin with some notation.
{\SB Recall $\XX \doteq  \clc([0,\infty): \RR_+^{\infty})$.} 
Define measurable maps $\{\theta_t\}_{t\ge 0}$ from  $(\XX, \clb(\XX))$ to itself as
$$\theta_t(\Zbd)(s) \doteq \Zbd(t+s), \; t\ge 0, s\ge 0, \; \Zbd \in \XX.$$
Given $\gbd \in \cld$ and $\zbd \in \clv$, we denote the probability distribution of the gap process of the $\gbd$-Atlas model on $(\XX, \clb(\XX))$, with initial gap sequence $\zbd$, by
$\PP^{\gbd}_{\zbd}$. Also, for $\gamma \in \clp(\RR_+^{\infty})$ supported on $\clv$ (namely, $\gamma(\clv)=1$), let $\PP^{\gbd}_{\gamma} \doteq \int_{\RR_+^{\infty}} \PP^{\gbd}_{\zbd} \;\gamma(d\zbd)$. The corresponding expectation operators will be denoted as $\EE^{\gbd}_{\zbd}$ and $\EE^{\gbd}_{\gamma}$ respectively.
Denote by $\cli^{\gbd}$ the collection of all invariant (stationary) probability measures of the gap process of the $\gbd$-Atlas model supported on $\clv$, namely
$$ \cli^{\gbd} \doteq \{\gamma \in \clp(\XX):  \gamma(\clv^c)=0, \mbox{ and } \PP^{\gbd}_{\gamma} \circ \theta_t^{-1} = \PP^{\gbd}_{\gamma} \mbox{ for all } t \ge 0\}.$$
%

Abusing notation, the canonical coordinate process on $(\XX, \clb(\XX))$ will be denoted by $\{\Zbd(t)\}_{t\ge 0}$.
 Let $\clm(\RR_+^{\infty})$ be the collection of all real measurable maps on $\RR_+^{\infty}$. For $f \in \clm(\RR_+^{\infty})$, $\zbd \in \RR_+^{\infty}$ and $t\ge 0$ such that $\EE^{\gbd}_{\zbd}(|f(\Zbd(t))|) <\infty$
 we write
 \begin{equation}\label{eq:semgpdep}
	 T_t^{\gbd}f(\zbd) \doteq \EE^{\gbd}_{\zbd}(f(\Zbd(t))).\end{equation}
Note that for $\gbd \in \cld$, $\gamma \in \cli^{\gbd} $, and
$\psi \in L^2(\gamma)$, $T_t^{\gbd}\psi$ is $\gamma$ a.e. well defined and belongs to $L^2(\gamma)$. Furthermore, the collection $\{T_t^{\gbd}\}_{t\ge 0}$ defines a contraction semigroup on $L^2(\gamma)$, namely
$$T_t^{\gbd}T_s^{\gbd} \psi = T_{t+s}^{\gbd}\psi, \mbox{ and } \|T_t^{\gbd}\psi\| \le \|\psi\| \mbox{ for all } s,t \ge 0 \mbox{ and } \psi \in L^2(\gamma).$$

We now recall the definition of extremality and ergodicity. 
Let, for $\gbd$  as above and $\gamma \in \cli^{\gbd}$, $\II_{\gamma}^{\gbd} $ be the collection of all $T_t^{\gbd}$-invariant square integrable functions, namely,
$$\II_{\gamma}^{\gbd} \doteq \{\psi \in L^2(\gamma): T_t^{\gbd}\psi = \psi,\;  \gamma  \mbox{ a.s., for all } t \ge 0\}.$$
We denote the projection of a $\psi \in L^2(\gamma)$ on to the closed subspace $\II_{\gamma}^{\gbd}$ as $\hat \psi_{\gamma}^{\gbd}$. {\SB 
Namely, $\hat \psi_{\gamma}^{\gbd}$ is the unique element of $\II_{\gamma}^{\gbd}$ that satisfies
$$\langle \psi, \eta \rangle = \langle \hat \psi_{\gamma}^{\gbd}, \eta \rangle, \mbox{ for all } \eta \in \II_{\gamma}^{\gbd}.$$
This projection can be obtained as the limit of $\frac{1}{t} \int_0^t T_s^{\gbd} \psi \,  ds$ in $L^2(\gamma)$ as $t \rightarrow \infty$ (see \eqref{eq:334}). Thus, for any $\psi \in L^2(\gamma)$, $\hat \psi_{\gamma}^{\gbd}(\cdot)$ can be intuitively interpreted as the `long-time average' of $\{\EE^{\gbd}_{\cdot}(\psi(\Zbd(t)) : t \ge 0\}$.}
\begin{definition}
	Let $\gbd\in \cld$. A  $\nu \in \cli^{\gbd}$ is said to be an extremal invariant distribution of the gap process of the $\gbd$-Atlas model if, whenever for some $\veps \in (0,1)$ and  $\nu_1, \nu_2 \in \cli^{\gbd}$  we have $\nu = \veps \nu_1 + (1-\veps)\nu_2$, then $\nu_1 = \nu_2 = \nu$.
	We denote the collection of all such measures by $\cli^{\gbd}_e$.
	
	We call $\nu \in \cli^{\gbd}$ an ergodic probability measure for the invariant distribution of the gap process of the $\gbd$-Atlas model if for all
	$\psi \in L^2(\nu)$,  $\hat \psi_{\nu}^{\gbd}$ is constant $\nu$-a.s. We denote the collection of all such measures by $\cli^{\gbd}_{er}$.
\end{definition}
We note that (cf. proof of Lemma \ref{lemerg} below) if $\gamma \in \cli^{\gbd}_{er}$, then for any $\psi \in L^2(\gamma)$, 
$$\frac{1}{t} \int_0^t T_s^{\gbd}\psi(\cdot) ds \to \int_{\RR_+^{\infty}} \psi(\xbd) \gamma(d\xbd), \mbox{ in } L^2(\gamma), \mbox{ as } t \to \infty.$$

The following result, which says that extremal invariant measures and ergodic invariant measures are the same, is standard, however we provide a proof in the appendix for completeness.
\begin{lemma}\label{lemerg}
	Let $\gbd \in \cld$. Then $\cli^{\gbd}_e = \cli^{\gbd}_{er}$.
	Let $\gamma \in \cli^{\gbd}$ and suppose that for every bounded measurable $\psi:\RR_+^{\infty} \to \RR$, 
	$\hat \psi_{\gamma}^{\gbd}$ is constant, $\gamma$ a.s. Then
	$\gamma \in \cli^{\gbd}_e$.
\end{lemma}

The following is the first main result of this work.
\begin{theorem}\label{thm_main}
	Let $\gbd\in \cld$. Then, for every $a > -2 \inf_{n \in \NN} \bar g_n$, $\pi_a^{\gbd} \in \cli^{\gbd}_e = \cli^{\gbd}_{er}$.  Furthermore, when $\gbd \in \cld_1$, $\pi_a^{\gbd} \in \cli^{\gbd}_e = \cli^{\gbd}_{er}$ also for $a = -2 \inf_{n \in \NN} \bar g_n$.
\end{theorem}
The above theorem proves the extremality of the invariant measures $\pi_a^{\gbd}$ for suitable values of $a$. 
{\SB As an immediate consequence of this theorem one can identify natural collections of measures that are in the (weak) domain of attraction of a given $\pi^{\gbd}_a$, as noted in the corollary below.
}
{\SB 
{\SB We recall that a measure $\gamma \in \clp(\RR_+^{\infty})$ is said to be in the weak (resp. strong) domain of attraction of $\pi^{\gbd}_a$ if for any bounded continuous function $\psi: \RR_+^{\infty} \rightarrow \RR$,
$$
\frac{1}{t}\int_0^t \EE^{\gbd}_{\gamma}(\psi(\Zbd(s)))ds \to \int_{\RR_+^{\infty}} \psi(\xbd) \pi^{\gbd}_a(d\xbd), 
$$
(resp. $\EE^{\gbd}_{\gamma}(\psi(\Zbd(t))) \to \int_{\RR_+^{\infty}} \psi(\xbd) \pi^{\gbd}_a(d\xbd)$), as $t \to \infty$.}
\begin{corollary}\label{doacor}
Let $\gbd\in \cld$. Fix any $a > -2 \inf_{n \in \NN} \bar g_n$. Let $\gamma \in \clp(\RR_+^{\infty})$ be absolutely continuous with respect to $\pi^{\gbd}_a$. Then $\gamma$ lies in the weak domain of attraction of $\pi^{\gbd}_a$. The assertion holds for any $a \ge -2 \inf_{n \in \NN} \bar g_n$ if $\gbd \in \cld_1$.
\end{corollary}
Sufficient conditions for a probability measure to be in the strong domain of attraction of $\pi^{\gbd^1}_0$ were obtained in 
\cite{AS, DJO} whereas weak domain of attraction results for $\pi^{\gbd^1}_a$, $a \ge 0$, have been obtained in \cite{banerjee2022domains}. The above corollary provides a weak domain of attraction result for 
 a  general class of $\gbd$-Atlas models.
}

One can ask whether these are the only extremal invariant measures of the gap process of the $\gbd$-Atlas model supported on $\clv$. As noted in the Introduction, the answer to this question when $\gbd = \zero$ is affirmative from results of \cite[Theorem 4.2]{RuzAiz} , if one restricts to extremal measures satisfying certain  integrability constraints on the denseness of particle configurations.
For a general $\gbd \in \cld$ (in fact even for $\gbd = \gbd^1$) this is currently a challenging open problem. However, we make partial progress towards this goal in the next result by showing that for any $\gbd \in \cld_1$ (and under a mild integrability condition), the collection $\{\pi_a^{\gbd}\}$ exhausts all the extremal product form invariant distributions. In fact we prove the substantially stronger statement that  the measures $\pi_a^{\gbd}$ are the only product form (extremal or not) invariant distributions under a mild integrability condition. Qualitatively, this result says  that these measures are  the only invariant distributions that preserve independence of the marginal laws of the gaps in time.

\begin{theorem}\label{thm_main2}
	Let $\gbd \in \cld_1$ and let $\pi \in \cli^{\gbd}$ be a product measure, i.e. $\pi = \otimes_{i=1}^{\infty} \pi_i$ for some $\pi_i \in \clp(\RR_+)$, $i \in \NN$. {\SB With $F(\zbd) \doteq \sum_{j=1}^{\infty} e^{-\frac{1}{4} (\sum_{l=1}^j z_l)^2}, \, \zbd \in \RR_+^{\infty}$, suppose that
    \begin{equation}\label{eq:satinteg}
        \int_{\RR_+^{\infty}} F(\zbd) \pi(d\zbd) <
        \infty.
    \end{equation}}
    Then, for some {\SB $a\ge -2 \inf_{n \in \NN} \bar g_n$}, $\pi = \pi_a^{\gbd}$. {\SB Moreover, $a$ has the representation
 \begin{equation}\label{arep}
 a = \EE^{\gbd}_{\pi}(L^*_1(1)) - 2g_0 = \frac{\EE^{\gbd}_{\pi}(L^*_k(1))}{k} - \frac{2}{k}(g_0 + \dots + g_{k-1}),
 \end{equation}
 for any $k \in \NN$, where $\{L^*_i\}_{i \in \NN}$ denote the collision local times in \eqref{order_SDE}.}
\end{theorem}
 Recall that $\clv$ defined in \eqref{clvdef} consists of $\zbd \in \RR_+^{\infty}$ for which $\sum_{j=1}^{\infty}e^{-\alpha (\sum_{l=1}^j z_l)^2} < \infty$ for all $\alpha>0$. In comparison, condition \eqref{eq:satinteg} requires a finite expectation of $\sum_{j=1}^{\infty}e^{-\frac{1}{4}(\sum_{l=1}^j z_l)^2}$ when $\zbd$ is distributed as $\pi$. 
Roughly speaking, condition  \eqref{eq:satinteg}
 puts a restriction on the rate of increase of the density of particles as one moves away from the lowest ranked particle.
 
 Several remarks are in order.
 
\begin{remark}[Probabilistic interpretation of $a$]\label{arem}
 The equalities \eqref{arep} give a probabilistic interpretation to $a$. By stationarity of $\pi$, $\EE^{\gbd}_{\pi}(L^*_1(1))$ can be thought of as the expected rate of change of the local time $L^*_1$. Hence, $\frac{\EE^{\gbd}_{\pi}(L^*_1(1))}{2}$ is intuitively the expected rate at which the bottom particle is `pushed down' by the particle above it during collisions and $g_0$ denotes its upward drift in time. Thus, $a/2$ captures the difference between two kinds of forces acting on the bottom particle: the \emph{hardcore interactions} due to collisions and the \emph{soft potential} corresponding to the drift. For $k>1$, one obtains a similar interpretation as follows. 
 Consider the subsystem consisting of the $k$ lowest particles viewed as a rank based diffusion $(Y^k_0,\dots, Y^k_{k-1})$, where $Y^k_i$ gets upward drift $g_j$ if its rank in the subsystem is $j$, and it is reflected downwards when it collides with the minimum of the particles outside the subsystem. It can be deduced that each particle $Y^k_i$ accrues roughly the same proportion of local time due to downward reflection as time grows. Moreover, it asymptotically spends an equal proportion of time at each rank $j \in \{0,\dots,k-1\}$.
 Hence, $\frac{\EE^{\gbd}_{\pi}(L^*_k(1))}{2k}$ and $\frac{1}{k}(g_0 + \dots + g_{k-1})$ respectively quantify the effect of reflection and drift on each particle among the lowest $k$ particles, and $a/2$ captures the difference between these effects.
 The positivity of $a$ implies that the hardcore interactions dominate in the long term. Indeed, the results of \cite{tsai_stat} show that when $\gbd = \gbd^1$, under $\pi^{\gbd^1}_a$, for any $k \in \NN$, $Y_{(k)}(t)/t \rightarrow -a/2$ almost surely as $t \rightarrow \infty$.
 We conjecture that the same result is true for any $\gbd \in \cld_1$.

 \end{remark}

\begin{remark}
When $\gbd \notin \cld$, uniqueness in law for the infinite-dimensional gap process is currently an open problem, and therefore  what one means by a stationary distribution is not clear. However, under conditions, for $\gbd \notin \cld$, one can still construct stationary `approximative' versions of this gap process by taking `limits' of finite-dimensional processes \cite[Definition 7]{AS} (see also Definition \ref{def:sttappver} below). See \cite[Theorem 4.4, Lemma 4.5 and Section 4.2]{AS} and \cite[Remark 3]{sarantsev2017stationary} for some situations where such versions can be constructed. Theorem \ref{thm_main2} can be extended to such settings as follows, as is clear from an inspection of the proof. Suppose $\gbd$ satisfies $\inf_{n \ge 1}\bar g_n > -\infty$ and there exist $N_1 < N_2 < \dots \rightarrow \infty$ such that $\bar g_k > \bar g_{N_j}$ for all $k=1,\dots,N_j-1, \ j \ge 1$. If there is a stationary approximative version of the infinite-dimensional gap process with initial (invariant) distribution $\pi$ supported on $\clv$, and if $\pi$ is a product measure that satisfies the integrability property in \eqref{eq:satinteg},  then $\pi = \pi_a^{\gbd}$ for some $a \ge  -2\lim_{j \rightarrow \infty} \bar g_{N_j}$.
\end{remark}

\begin{remark}
Stationary distributions for finite dimensional reflected Brownian motions (with drift) in the positive orthant,
have been studied in \cite{harrison1987multidimensional, harrison1987brownian}. In particular, the paper \cite{ harrison1987brownian}  shows that the unique stationary distribution can be characterized through an identity, holding true for all suitable smooth test functions, referred to as the Basic Adjoint Relationship (BAR) (see \cite[Section 8]{harrison1987brownian}). Using this characterization it is shown in \cite[Section 9]{ harrison1987brownian}  that if the stationary distribution is product form then it must necessarily be a product of Exponential distributions.
The proof relies on using a suitable class of exponential test functions in the BAR to characterize the moment generating function of the stationary distribution. In the infinite dimensional setting considered here, although similar test functions are useful, we do not know of a similar BAR characterization for all stationary distributions. To circumvent this, we show in Lemma \ref{lem:lociden} that for any product form stationary distribution, one can obtain  `local' descriptions for the expectations of certain path functionals of the process. This result is key and essentially plays the role of BAR in our context in obtaining a recursive system of equations for the moment generating functions of the marginal distributions that can then be solved explicitly to prove Theorem \ref{thm_main2}.
\end{remark}

\begin{remark}
A referee has proposed the following interesting direction of investigation. Suppose that $\gamma \ll \pi_b^{\gbd}$ for some $b \ge -2 \inf_{n \in \NN} \bar g_n$. Then one may conjecture that, for each $k \in \NN$, 
$$\frac{\EE^{\gbd}_{\gamma}(L^*_k(t))}{kt} - \frac{2}{k}(g_0 + \dots + g_{k-1}) \to b\mbox{ as } t \to \infty.$$
One may also ask the following `domain of attraction' question. Given $b \ge -2 \inf_{n \in \NN} \bar g_n$, identify the set $\clv_b \subset \clv$ such that for $z \in \clv_b$, for each $k \in \NN$,
$$\frac{\EE^{\gbd}_{z}(L^*_k(t))}{kt} - \frac{2}{k}(g_0 + \dots + g_{k-1}) \to b \mbox{ as } t \to \infty.$$
In this case, we conjecture that any $\gamma \in \clp(\RR_+^{\infty})$ supported on $\clv_b$ is in the strong domain of attraction of $\pi_b^{\gbd}$. We leave the study of these questions for future work.
\end{remark}

{\SB Rest of the paper is devoted to the proofs of Theorem \ref{thm_main} and Theorem \ref{thm_main2}. Proof of Lemma \ref{lemerg} is given in the Appendix for completeness.}

\section{Proof of Theorem \ref{thm_main}}
\label{sec:prffirstthm}
We will only prove the first statement in Theorem \ref{thm_main}. The proof of the second is similar and is therefore omitted.

We begin with the following definition. Let
$\YY \doteq \clc([0, \infty): \RR_+^{\infty} \times \RR_+^{\infty})$.
\begin{definition}
	Let $\gbd\in \cld$ and $\gamma, \gamma' \in \clp(\RR_+^{\infty})$ be such that $\gamma(\clv)= \gamma'(\clv)=1$. 
	We say that a probability measure $\PP^{\gbd}_{\gamma, \gamma'}$  on $(\YY, \clb(\YY))$ defines
	 a coupling of the gap process of the $\gbd$-Atlas model with initial distributions $\gamma, \gamma'$, if, denoting the coordinate processes on $\YY$ as $\Zbd^{(1)}$ and $\Zbd^{(2)}$, namely
	 $$\Zbd^{(1)}(\om)(t) \doteq \om^{(1)}(t), \;\; \Zbd^{(2)}(\om)(t) \doteq \om^{(2)}(t), \; \om = (\om^{(1)}, \om^{(2)})\in \YY, \; t \ge 0,$$
	 we have 
	 $$\PP^{\gbd}_{\gamma, \gamma'} \circ (\Zbd^{(1)})^{-1} = \PP^{\gbd}_{\gamma}, \;\; \PP^{\gbd}_{\gamma, \gamma'} \circ (\Zbd^{(2)})^{-1} = \PP^{\gbd}_{\gamma'}.$$
Define the coupling time
$$
\tau_{c} \doteq \inf\{t \ge 0: \Zbd^{(1)}(s) = \Zbd^{(2)}(s) \text{ for all } s \ge t\},
$$
where $\tau_c \doteq \infty$ if the above set is empty.
When $\gamma = \delta_{\zbd}$ and $\gamma' = \delta_{\zbd'}$ for some $\zbd, \zbd' \in \clv$, we write
	  $\PP^{\gbd}_{\gamma, \gamma'} = \PP^{\gbd}_{\zbd, \zbd'}$.
\end{definition}
Since $\gbd \in \cld$ will be fixed throughout the section, we will frequently suppress it from the notation.

Consider now  $\gamma = \pi_a^{\gbd}$, where $a$ is as in the statement of the theorem,  and a bounded measurable map $\psi_0: \RR_+^{\infty} \to \RR$ such that
\begin{equation} \label{eq:1228} T_t \psi_0 = \psi_0,  \; \gamma \mbox{ a.s. for every } t \ge 0,\end{equation}
{\SB where $T_t$ is defined as in \eqref{eq:semgpdep}, namely,  $T_t\psi_0(\zbd)$ is defined for $\gamma$ a.e. $z$ as $T_t\psi_0(\zbd)= \EE_{\zbd}(\psi_0(\Zbd(t)))$,  $t \ge 0$.}
In order to prove Theorem \ref{thm_main} it suffices, in view of Lemma \ref{lemerg}, to show that
 $\psi_0$ is $\gamma$-a.e. constant.
This, in view of \eqref{eq:1228}, is equivalent to showing that for some fixed $t_0>0$, $\psi \doteq T_{t_0}\psi_0$ is $\gamma$-a.e. constant. For the rest of the section we will fix a $t_0>0$ and consider $\psi$ defined as above.  Note that
\eqref{eq:1228} holds with $\psi_0$ replaced by $\psi$.
{\SB\subsection{Proof overview }\label{sec:pfov1}
Before we proceed to the details, we  give a brief overview of the proof strategy for showing that $\psi$ is $\gamma$-a.e. constant. The first step is to show using the $T_t$-invariance of $\psi$ that for any $t \ge 0$, $\psi(\Zbd(t)) = \psi(\zbd)$ for $\gamma$-a.e. $\zbd$. Moreover, using the product form of $\gamma$, the same conclusion is seen to hold for the process $\Zbd(\cdot)$ started from a `perturbed' initial point obtained by changing any two co-ordinates of $\gamma$-a.e. $\zbd$ by given numbers (see \eqref{eq:areequal}). Up to this point, we only use quite general arguments not involving the specific dynamics of the $\gbd$-Atlas model. However, the dynamics comes into play crucially in the subsequent steps, which involve construction of a coupling of two $\gbd$-Atlas models started from initial points that differ at a finite number of co-ordinates. This  is achieved by a combination of  `mirror'  and synchronous couplings of the infinite collection of driving Brownian motions (see \eqref{mirror} and \eqref{eq:coup}). The coupling is utilized in two ways. First, it is shown that for any $s>0$, the coupled $\gbd$-Atlas models coalesce with positive probability by time $s$ (Proposition \ref{prop:prop1}). Using this and \eqref{eq:areequal}, it follows that the value of $\psi$ remains unchanged upon changing any pair of coordinates by rational numbers (see \eqref{eq:gamone}). To extend this to perturbation by real numbers (see \eqref{630a} and \eqref{630b}), we need a key `directional strong Feller property' described in Proposition \ref{prop:prop2}, which is once again established using the coupling. The equality of $\psi$ under pairwise perturbations is then extended to perturbation by any finite permutation via straightforward algebraic manipulations. The proof of $\gamma$ almost sure constancy of $\psi$, and hence of Theorem \ref{thm_main}, is finally achieved by an application of the Hewitt-Savage zero-one law.}
\subsection{Preliminary Results}
Now, we proceed to the details. We begin by noting that, from \eqref{eq:1228} (with $\psi_0$ replaced by $\psi$),
\begin{align*}
	0 &= \EE_{\gamma}\left(\psi(\Zbd(t))^2\right) - \EE_{\gamma}\left(\psi(\Zbd(0))^2\right) =
	\EE_{\gamma}\left(\psi(\Zbd(t))^2\right) + \EE_{\gamma}\left(\psi(\Zbd(0))^2\right) - 2 
	\EE_{\gamma}\left(\psi(\Zbd(0))\psi(\Zbd(t))\right)\\
	& =\EE_{\gamma}\left(\psi(\Zbd(t)) - \psi(\Zbd(0))\right)^2.
\end{align*}
This says 
\begin{equation}\label{eq:346}
	\psi(\Zbd(t)) = \psi(\Zbd(0)), \; \PP_\gamma \mbox{ a.s., for every  } t \ge 0.
\end{equation}
Next let
{\SB\begin{equation}
	H(\zbd,t) \doteq \EE_{\zbd}|\psi(\Zbd(t)) - \psi(\Zbd(0))|, \; \zbd \in \clv, \, t \ge 0.
\end{equation}}
For $i \in \NN$,  $x,y \in \RR_+$, and $\bar \gamma \in \clp(\clv)$ of the form $\bar \gamma = \otimes_{i=1}^{\infty} \bar\gamma_i$ for some $\bar\gamma_i \in \clp(\RR_+)$, $i \in \NN$, define
\begin{equation}\label{eq:etadef}
	\eta_i^{\bar\gamma}(x,y,t) \doteq \int_{\RR_+^{\infty}} H(z_1, \ldots, z_{i-1}, x,y, z_{i+2}, \ldots,t) \prod_{j \in \NN \setminus \{i,i+1\}}\bar\gamma_j(dz_j).
\end{equation}
We have from the Markov property that 
\begin{equation} \label{eq:gambar}
	\EE_{\bar\gamma}\left| \psi(\Zbd(t)) - \psi(\Zbd(0))\right|	= \int_{\RR_+^2 }\eta_i^{\bar\gamma}(x,y,t) \bar \gamma_{i} (dx)\bar \gamma_{i+1} (dy).
\end{equation}

Now take $\bar \gamma=\gamma = \pi_a^{\gbd} = \otimes_{i=1}^{\infty} \gamma_i$.
Then, from \eqref{eq:346},
\begin{align}\label{eq:firstisze}
0&= 	\EE_{\gamma}\left|\psi(\Zbd(t)) - \psi(\Zbd(0))\right | = \int_{\RR_+^2} \eta_i^{\gamma}(x,y,t) \gamma_i(dx)\gamma_{i+1}(dy).
\end{align}
Recall that, for each $i \in \NN$, $\gamma_i$ is an Exponential distribution and thus is mutually absolutely continuous with respect to the Lebesgue measure $\lambda$ on $\RR_+$.
Since $\eta_i^{\gamma}$ is nonnegative, we have from this mutual absolute continuity property that, {\SB for any $i \in \NN$,}
\begin{equation}\label{eqizzo}
	\eta_i^{\gamma}(x,y,t) = 0,\; \lambda \otimes \lambda \mbox{ a.e. } (x,y) \in \RR_+^2, \, \mbox{ for every  } t \ge 0.
\end{equation}
Fix $\delta_1,\delta_2>0$, {\SB $i \in \NN$, $t \ge 0$}. For $y \ge 0$, let $\varsigma_{\delta_2}(y) \doteq \delta_2 1_{y>\delta_2}$. Letting $\Zbd$ be a $\RR_+^{\infty}$-valued random variable distributed as $\gamma$, denote by
$\tilde \gamma \in \clp(\RR_+^{\infty})$ the probability distribution of
$\tilde \Zbd \doteq \Zbd + \delta_1 \ebd_i - \varsigma_{\delta_2}(Z_{i+1}) \ebd_{i+1}$ where
$\ebd_i$ is the unit vector in $\RR_+^{\infty}$ with $1$ at the $i$-th coordinate.
Note that by definition in \eqref{eq:etadef}, $\eta_i^{\gamma} = \eta_i^{\tilde\gamma}$, and in particular, from \eqref{eqizzo},
\begin{equation}
	\eta_i^{\tilde\gamma}(x+\delta_1, y-\varsigma_{\delta_2}(y),t) = \eta_i^{\gamma}(x+\delta_1, y-\varsigma_{\delta_2}(y),t)=0,\; \lambda \otimes \lambda \mbox{ a.e. } (x,y) \in \RR_+^2.
\end{equation}
Thus, in view of \eqref{eq:gambar},
\begin{equation}\label{eq:secisze}
	\EE_{\tilde\gamma}\left|\psi(\Zbd(t)) - \psi(\Zbd(0))\right | = \int_{\RR_+^2 }\eta_i^{\tilde\gamma}(x,y,t) \tilde \gamma_{i}(dx) \tilde \gamma_{i+1}(dy)=
	\int_{\RR_+^2} \eta_i^{\tilde\gamma}(x+\delta_1, y-\varsigma_{\delta_2}(y),t)\gamma_i(dx)\gamma_{i+1}(dy)=0,
\end{equation}
where we have used the fact that $\gamma_i\otimes \gamma_{i+1}$ is mutually absolutely continuous with respect to $\lambda \otimes \lambda$.
For $\zbd \in \clv$, let $\beta_{\delta_1,\delta_2}(\zbd) \doteq \zbd + \delta_1 \ebd_i - \varsigma_{\delta_2}(z_{i+1}) \ebd_{i+1}$.
Then, combining \eqref{eq:firstisze} and \eqref{eq:secisze}, we have
\begin{equation}
\EE_{\zbd}\left|\psi(\Zbd(t)) - \psi(\Zbd(0))\right | = \EE_{\beta_{\delta_1,\delta_2}(\zbd)}\left|\psi(\Zbd(t)) - \psi(\Zbd(0))\right | = 0, \; \gamma \mbox{ a.e. } \zbd, \mbox{ for every } t \ge 0.
\end{equation}
Since $\beta_{\delta_1,\delta_2}(\zbd) = \zbd + \delta_1 \ebd_i - \delta_2\ebd_{i+1} \doteq \zbd^{\delta_1,\delta_2,i}$ when $z_{i+1}> \delta_2$, we get\
\begin{equation}\label{eq:areequal}
\psi(\Zbd(t)) = \psi(\zbd),\,  \PP_{\zbd} \mbox{ a.e., }
\psi(\Zbd(t)) = \psi(\zbd^{\delta_1,\delta_2,i}),\,  \PP_{\zbd^{\delta_1,\delta_2,i}} \mbox{ a.e., for } \; \gamma \mbox{ a.e. } \zbd \mbox{ with } z_{i+1}>\delta_2, \mbox{ for every } t \ge 0.
\end{equation}

We will need the following proposition. The proof is given in Section \ref{pfprop1}.
\begin{proposition}
	\label{prop:prop1}
	Fix $i\in \NN$, $\zbd \in \clv$ with $\zbd>0$, and $\delta_1>0, \delta_2 \in (0, z_{i+1})$. Then there exists a coupling $\PP_{\zbd, \delta_1,\delta_2,i}$
	of    the gap process of the $\gbd$-Atlas model with initial distributions $\delta_{\zbd}$ and
	$\delta_{\zbd^{\delta_1,\delta_2,i}}$ such that, for any $s>0$,
	$$\PP_{\zbd, \delta_1,\delta_2,i}(\Zbd^{(1)}(s) = \Zbd^{(2)}(s)) >0.$$
\end{proposition}
Now, for $\delta_1,\delta_2>0$ and $s>0$,
\begin{align*}
	&\int_{\{\zbd: z_{i+1}>\delta_2\}} |\psi(\zbd^{\delta_1,\delta_2,i}) - \psi(\zbd)| \PP _{\zbd, \delta_1,\delta_2,i}(\Zbd^{(1)}(s) = \Zbd^{(2)}(s)) \gamma(d\zbd)\\
	&= \int_{\{\zbd: z_{i+1}>\delta_2\}} \EE _{\zbd, \delta_1,\delta_2,i}\left[|\psi(\zbd^{\delta_1,\delta_2,i}) - \psi(\zbd)| 1_{\{\Zbd^{(1)}(s) = \Zbd^{(2)}(s)\}}\right] \gamma(d\zbd)\\
	&=\int_{\{\zbd: z_{i+1}>\delta_2\}} \EE _{\zbd, \delta_1,\delta_2,i}\left[|\psi(\Zbd^{(2)}(s)) - \psi(\Zbd^{(1)}(s))| 1_{\{\Zbd^{(1)}(s) = \Zbd^{(2)}(s)\}}\right] \gamma(d\zbd) =0,
\end{align*}
where $\Zbd^{(i)}$ are as given by Proposition \ref{prop:prop1} and
the second equality follows from \eqref{eq:areequal}.
Hence, from Proposition \ref{prop:prop1}, for every $\delta_1,\delta_2>0$ and $i \in \NN$,
$$\psi(\zbd +  \delta_1 \ebd_i - \delta_2\ebd_{i+1}) = \psi(\zbd) \mbox{ for } \gamma \mbox{ a.e. } \zbd \mbox{ with } z_{i+1} >\delta_2.$$
Thus we have shown that 
\begin{equation}\label{eq:241}
	\gamma(\cup_{\delta_1,\delta_2 \in \QQ \cap (0,\infty)}\{\zbd: z_{i+1} >\delta_2, \mbox{ and } \psi(\zbd +  \delta_1 \ebd_i - \delta_2\ebd_{i+1}) \neq \psi(\zbd)\}) =0.
\end{equation}
This implies that 
\begin{equation}\label{eq:gamone}
\gamma(\zbd: \psi(\zbd +  \delta_1 \ebd_i - \delta_2\ebd_{i+1}) = \psi(\zbd) \mbox{ for all } \delta_1 \in (0,\infty) \cap \QQ, \ \delta_2 \in (0, z_{i+1}) \cap \QQ)=1.
\end{equation}
To see this, let $B$ denote the event on the left side of \eqref{eq:gamone}. Then if $\zbd \in B^c$ and $z_{i+1}>0$, then for some $\delta_1, \delta_2 \in (0,\infty) \cap \QQ$, $z_{i+1}>\delta_2$ and $\psi(\zbd +  \delta_1 \ebd_i - \delta_2\ebd_{i+1}) \neq \psi(\zbd) $, which shows that $\zbd$ is in the event on the left side of \eqref{eq:241} which in view of \eqref{eq:241} says that $\gamma(B^c)=0$, proving the statement in \eqref{eq:gamone}.
The following proposition enables us to extend \eqref{eq:gamone} to \emph{all} $(\delta_1,\delta_2) \in (0,\infty) \times (0,z_{i+1})$.
The proof is given in Section \ref{pfprop1}. 
\begin{proposition}
	\label{prop:prop2}
	For each $\zbd \in \clv$ with $\zbd>0$ and $i \in \NN$,  the map $(\delta_1,\delta_2) \mapsto \psi(\zbd + \delta_1 \ebd_i - \delta_2\ebd_{i+1})$ is right continuous on $[0,\infty) \times [0, z_{i+1})$. That is, if $(\delta_1, \delta_2) \in [0,\infty) \times [0, z_{i+1})$ and $\delta_{1,n} \downarrow \delta_1$ and $\delta_{2,n} \in [0, z_{i+1})$ with $\delta_{2,n} \downarrow \delta_2$ as $n \rightarrow \infty$, then $\psi(\zbd + \delta_{1,n} \ebd_i - \delta_{2,n}\ebd_{i+1}) \rightarrow \psi(\zbd + \delta_1 \ebd_i - \delta_2\ebd_{i+1})$ as $n \rightarrow \infty$.
\end{proposition}
We remark that our proof shows that in the above proposition $\psi$ can be taken to be $T_{t_0}\psi_0$ for any real bounded measurable  function $\psi_0$ on $\RR_+^{\infty}$ and $t_0>0$, namely it need not be $\{T_t\}$-invariant. Thus the property established in the above proposition can be viewed as a certain type of `directional strong Feller property'.

\subsection{Completing the proof of Theorem \ref{thm_main}}
As an immediate consequence of the above proposition and \eqref{eq:gamone} we have that
for  $\gamma$  a.e. 	$\zbd$,
\begin{equation}\label{630a}
 \psi(\zbd + \delta_1 \ebd_i - \delta_2\ebd_{i+1}) = \psi(\zbd) \mbox{ for all } \delta_1 \in (0,\infty), \delta_2 \in (0, z_{i+1}), \; i \in \NN.
\end{equation}
A similar argument shows that, for  $\gamma$  a.e. 	$\zbd$,
\begin{equation}\label{630b}
\psi(\zbd - \delta_1 \ebd_i + \delta_2\ebd_{i+1}) = \psi(\zbd) \mbox{ for all } \delta_1 \in (0, z_{i}), \delta_2 \in (0,\infty), \; i \in \NN.
\end{equation}
We now proceed to the proof of the first statement in Theorem \ref{thm_main}. Recall that $\gamma = \pi_a^{\gbd} \doteq \otimes_{n=1}^{\infty} \mbox{Exp}(n (2\bar g_n+a))$, where $a$ satisfies the condition in Theorem \ref{thm_main}. For notational simplicity, let $c_n \doteq 2[n (2\bar g_n+a)]^{-1}$, $n \in \NN$.
Let $\tilde\psi: \clv \to \RR$ be defined as
$$\tilde\psi(z_1, z_2, z_3, \ldots) \doteq \psi(c_1z_1, c_2z_2, c_3z_3, \ldots), \; \zbd = (z_1, z_2, \ldots) \in \clv$$
if $(c_1z_1, c_2z_2, c_3z_3, \ldots) \in \clv$. For all other $\zbd \in \clv$, set $\tilde \psi(\zbd)=0$.
We denote $\pi_0^{\gbd^1}\doteq \otimes_{n=1}^{\infty} \mbox{Exp}(2)$ as $\pi_0$ for simplicity. 
Observe that, for any $i \in \mathbb{N}_0$,
\begin{align}\label{eq:909}
&\pi_0(\zbd: \tilde \psi(z_1, z_2, \ldots , z_{i-1}, z_i, z_{i+1}, z_{i+2}, \ldots) = \tilde \psi(z_1, z_2, \ldots , z_{i-1}, z_{i+1}, z_{i}, z_{i+2}, \ldots))\notag\\
&\hspace{-1cm}=\pi_a^{\gbd}\left(\zbd:  \psi(z_1, z_2, \ldots , z_{i-1}, z_i, z_{i+1}, z_{i+2}, \ldots) = \psi\left(z_1, z_2, \ldots , z_{i-1}, \frac{c_i}{c_{i+1}}z_{i+1}, \frac{c_{i+1}}{c_{i}}z_{i}, z_{i+2}, \ldots\right)\right).
\end{align}
Consider the set $C \in \clb(\XX)$ with $\pi_a^{\gbd}(C)=1$ on which the  two statements in \eqref{630a} and \eqref{630b} hold. Then for any $\zbd \in C$ such that $z_i>0$ for all $i\in \NN$, we have,
\begin{equation}\psi(z_1, z_2, \ldots , z_{i-1}, z_i, z_{i+1}, z_{i+2}, \ldots) = \psi\left(z_1, z_2, \ldots , z_{i-1}, \frac{c_i}{c_{i+1}}z_{i+1}, \frac{c_{i+1}}{c_{i}}z_{i}, z_{i+2}, \ldots\right).\end{equation}
Indeed, if $\frac{c_i}{c_{i+1}}z_{i+1}-z_i>0$, then the statement follows from \eqref{630a}  on taking $\delta_1 = \frac{c_i}{c_{i+1}}z_{i+1}-z_i$ and $\delta_2 =  \frac{c_{i+1}}{c_{i}} \delta_1$, and if 
$z_{i}-\frac{c_i}{c_{i+1}}z_{i+1} >0$, the statement follows from \eqref{630b} on taking $\delta_1 = z_{i}-\frac{c_i}{c_{i+1}}z_{i+1}$ and $\delta_2 =  \frac{c_{i+1}}{c_{i}} \delta_1$.
Since $\pi_a^{\gbd}(C)=1$, we have that the probability on the right side of \eqref{eq:909} is $1$ and so,
\begin{equation}
\pi_0(\zbd: \tilde \psi(z_1, z_2, \ldots , z_{i-1}, z_i, z_{i+1}, z_{i+2}, \ldots) = \tilde \psi(z_1, z_2, \ldots , z_{i-1}, z_{i+1}, z_{i}, z_{i+2}, \ldots))=1.
\end{equation}
{\SB As any finite permutation can be obtained as a composition of finitely many adjacent transpositions}, it now follows that, in fact for any finite permutation $\rho: \NN \to \NN$,
\begin{equation}
\pi_0(\zbd: \tilde \psi(z_1, z_2, \ldots) = \tilde \psi(z_{\rho(1)}, z_{\rho(2)}, \ldots))=1.
\end{equation}

Now using the Hewitt-Savage zero-one law (cf. \cite[Theorem 2.15]{kall}), $\tilde \psi$ is $\pi_0$ a.e. constant, namely, there is a $\alpha \in \RR$ such that
$\pi_0(\zbd: \tilde \psi(\zbd)=\alpha)=1$.
This shows that
$$\pi_a^{\gbd}(\zbd:  \psi(\zbd)=\alpha) = \pi_0(\zbd: \tilde \psi(\zbd)=\alpha)=1.$$
{\SB Hence, $\psi$ is constant $\pi_a^{\gbd}$ a.s. Appealing to Lemma \ref{lemerg}, this completes the proof of the first statement in Theorem \ref{thm_main}.} The second statement follows similarly. \hfill \qed

\subsection{Proofs of Propositions \ref{prop:prop1} and \ref{prop:prop2}}
\label{pfprop1}
Recall that we fix $\gbd \in \cld$. Also, throughout this section we fix $\zbd \in \clv$ such that $\zbd >0$ and $i \in \NN$.
Define for $\delta_2 \in (0, z_{i+1})$ and $\delta_1>0$
{\SB$$\ybd \doteq (0, z_1, z_1+z_2, \ldots)', \;\; \ybd^{\delta_1,\delta_2} \doteq \ybd + \sum_{j=0}^{i-1}(\delta_2-\delta_1)\ebd_j  + \delta_2\ebd_i.$$}
Observe that, with the above choice of starting points of the particles, the corresponding gaps are $\zbd(\ybd) = \zbd$ and $\zbd(\ybd^{\delta_1,\delta_2}) = \zbd^{\delta_1,\delta_2}$.

Let $B_0, B_1, \ldots$ be a sequence of mutually independent standard Brownian motions on some probability space $(\Om, \clf, \PP)$. Consider the $(i+1)$-dimensional diffusion process
$$
\Psib(t) = \Psib(0) + D\Bbd^{(i)}(t) + \bbd t, \ t \ge 0,
$$
where $\Psib(0) = (z_1,\dots, z_i, (z_{i+1} + \delta_2)/2)'$, $\Bbd^{(i)}(\cdot) = (B_0(\cdot),B_1(\cdot),\dots,B_i(\cdot))'$, $\bbd = (g_1-g_0,\dots,g_{i} - g_{i-1},-g_i)'$, and $D$ is an $(i+1) \times (i+1)$ matrix with $D_{jj} = -1$ for all $1 \le j \le i+1$, $D_{j (j+1)}=1$ for all $1 \le j \le i$ and $D_{jl}=0$ otherwise. The process $\{\Psib(t) : \, t \ge 0\}$ will be used to analyze the evolution of the first $i+1$ gaps before any of them hit zero or the lowest $i$ particles interact with the higher ranked particles (in an appropriate sense), as stated more precisely later.

It can be checked that $D$ is non-singular and $(D^{-1})_{jl} = -1$ for all $1 \le j \le l \le i+1$ and $(D^{-1})_{jl} = 0$ otherwise. Let $\tilde \Psib(0) \doteq (z_1,\dots,z_{i-1}, z_i + \delta_1,(z_{i+1} - \delta_2)/2)'$ and define {\SB$\vbd \doteq D^{-1}(\tilde\Psib(0) - \Psib(0)) = (\delta_2-\delta_1,\dots,\delta_2 - \delta_1,\delta_2)'$}. Also define the stopping time
$$
\sigma \doteq \inf\{t \ge 0: \, \vbd' \Bbd^{(i)}(t) = \|\vbd\|^2/2\},
$$
where $\|\cdot\|$ denotes the standard Euclidean norm.
Define the \emph{mirror coupled} $(i+1)$-dimensional Brownian motion $\tilde \Bbd^{(i)}$ by
\begin{equation}\label{mirror}
\tilde \Bbd^{(i)}(t) = \left(I - \frac{2\vbd \vbd'}{\|\vbd\|^2}\right)\Bbd^{(i)}(t)1_{[t \le \sigma]} + (\Bbd^{(i)}(t) - \vbd)1_{[t > \sigma]}, \ t \ge 0.
\end{equation}
Since $\left(I - \frac{2\vbd \vbd'}{\|\vbd\|^2}\right)$ is a unitary matrix, it follows from
the strong Markov property that $\tilde \Bbd^{(i)}$ is indeed a Brownian motion. {\SB$\tilde\Bbd^{(i)}$ can be thought of as the reflection of the Brownian motion $\Bbd^{(i)}$ in a hyperplane perpendicular to the vector $\vbd$ till the first time $\sigma$ when $\Bbd^{(i)}$ hits this hyperplane (which is also the first meeting time of $\Bbd^{(i)}$ and $\tilde\Bbd^{(i)}$), and then coalescing with $\Bbd^{(i)}$. Using $\tilde\Bbd^{(i)}$}, define a coupled version of the process $\Psib$ by
$$
\tilde\Psib(t) = \tilde \Psib(0) + D \tilde \Bbd^{(i)}(t) + \bbd t, \ t \ge 0.
$$
Extend $\tilde \Bbd^{(i)}$ to an infinite collection of standard Brownian motions $\tilde \Bbd = (\tilde B_0, \tilde B_1,\dots)' \doteq (\tilde \Bbd^{(i)}, B_{i+1},\dots)'$.

Our arguments will involve a coupling of two copies of the infinite ordered $\gbd$-Atlas model started from $\ybd$ and $\ybd^{\delta_1,\delta_2}$ and respectively driven by the Brownian motions $\{B_j\}_{j \in \NN_0}$ and $\{\tilde B_j\}_{j \in \NN_0}$. For the finite particle $\gbd$-Atlas model, this coupling can be directly constructed using the existence of a strong solution to the finite version of the SDE \eqref{order_SDE}. However, for the infinite $\gbd$-Atlas model, this is a delicate issue. We will use the recipe of \emph{approximative versions} of \cite{AS}, which we now introduce.
\begin{definition}
	\label{def:sttappver}
Suppose $\xbd\in \mathcal{U}$ and
consider a collection of independent standard Brownian motions $\{B^*_j\}_{j \in \NN_0}$. Consider for fixed $m \in \NN$, the system of SDE in \eqref{order_SDE} for $i=0, 1, \ldots, m$, with 
 starting configuration $X^*_i(0) = x_i$, $0 \le i \le m$, local times of collision between the $(i-1)$-th and $i$-th particles denoted by $L^*_i, 1\le i \le m$, and $L^*_{0}(\cdot) \equiv L^*_{m+1}(\cdot) \equiv 0$.
Denote by 
 $\Xbd^{*,(m)}(\cdot) =(X^{*,(m)}_0(\cdot),\dots, X^{*,(m)}_m(\cdot))'$ and $\Lbd^{*,(m)}(\cdot) =(L^{*,(m)}_0(\cdot),\dots, L^{*,(m)}_m(\cdot))'$
 the unique strong solution to this finite-dimensional system of reflected SDE with driving Brownian motions $(B^*_0(\cdot),\dots,B^*_m(\cdot))'$. 

Then (see  \cite[Definition 7 and Theorem 3.7]{AS}, \cite[Lemma 6.4]{AS} and the discussion following it), there exist continuous $\RR^{\infty}$-valued processes $\Xbd^*(\cdot) \doteq (X^*_i(\cdot) : i \in \mathbb{N}_0)'$,
 $\Lbd^*(\cdot) \doteq (L^*_i(\cdot) : i \in \mathbb{N}_0)'$, adapted to $\clf_t \doteq \sigma \{B^*_i(s): s\le t, i \in \NN_0\}$, such that, a.s. $\Xbd^*$ satisfies \eqref{order_SDE} with associated local times given by  $\Lbd^*$ and for any $T\in (0,\infty)$,
  $$
 \lim_{m \rightarrow \infty} \sup_{t \in [0,T]}\left[\left| X^{*,(m)}_i(t) - X^*_i(t)\right| +  \left| L^{*,(m)}_i(t) - L^*_i(t)\right|\right] = 0 \ \text{ a.s., for all } i \in \mathbb{N}_0.
 $$
 We will call $\Xbd^*(\cdot)$ the `infinite ordered $\gbd$-Atlas model' with driving Brownian motions $\{B^*_j\}_{j \in \NN_0}$ started from $\xbd= (x_0,x_1, x_2\dots)'$.
\end{definition}

We denote the infinite ordered $\gbd$-Atlas model defined on $(\Om, \clf, \PP)$ with initial condition $\ybd$ and driving Brownian motions $\{B_j\}_{j \in \NN_0}$ as $\Xbd = (X_0, X_1, \ldots)'$.
Similarly, denote the infinite ordered $\gbd$-Atlas model defined on $(\Om, \clf, \PP)$ with initial condition $\ybd^{\delta_1,\delta_2}$ and driving Brownian motions $\{\tilde B_j\}_{j \in \NN_0}$ as $\tilde \Xbd = (\tilde X_0, \tilde X_1, \ldots)'$.
Denote the gap processes associated with $\Xbd$ and $\tilde \Xbd$ as $\Zbd$ and $\tilde \Zbd$ respectively, namely
$$Z_i \doteq X_i - X_{i-1}, \;\; \tilde Z_i \doteq \tilde X_i - \tilde X_{i-1}, \;; i \in \NN.$$
It then follows that
\begin{equation}
\PP_{\zbd, \delta_1,\delta_2, i} \doteq \PP \circ (\Zbd, \tilde \Zbd)^{-1}  \label{eq:coup}
\end{equation}
is a coupling of the gap process of the $\gbd$-Atlas model with initial distributions $\delta_{\zbd}$ and
$\delta_{\zbd^{\delta_1,\delta_2,i}}$. Moreover, the process $\{\Psib(t) : \, t \ge 0\}$ gives the evolution of $\{\left(Z_1(t), \dots, Z_i(t), \frac{y_i + y_{i+1} + \delta_2}{2} - X_{i}(t)\right)' : \, t \ge 0\}$ before any of the co-ordinates of $\Psib(\cdot)$ hit zero or $X_{i+1}$ hits the level $\frac{y_i + y_{i+1} + \delta_2}{2}$ (note that $z_{i+1}>\delta_2$ guarantees that $X_{i+1}(0)>\frac{y_i + y_{i+1} + \delta_2}{2}$). This can be seen from \eqref{order_SDE}.
Similarly, $\{\tilde\Psib(t) : \, t \ge 0\}$ gives the evolution of $\{\left(\tilde Z_1(t), \dots, \tilde Z_i(t), \frac{y_i + y_{i+1} + \delta_2}{2} - \tilde X_i(t)\right)' : \, t \ge 0\}$ before any of the co-ordinates of $\tilde\Psib(\cdot)$ hit zero or $\tilde X_{i+1}$ hits the level $\frac{y_i + y_{i+1} + \delta_2}{2}$ from above.

We will now construct tractable events of positive probability under which the `mirror coupled' processes $\Psib$ and $\tilde \Psib$ will successfully couple before any of their co-ordinates hit zero or $X_{i+1}$ (equivalently, $\tilde X_{i+1}$) hits the level $\frac{y_i + y_{i+1} + \delta_2}{2}$. Towards this end, observe that $\mathcal{D} \doteq \{D^{-1} \xbd : \xbd \in \mathbb{R}^{i+1}_+\}$ is a polyhedral convex domain contained in the nonpositive orthant of $\mathbb{R}^{i+1}$. Let $\Lmc$ denote the line segment joining $D^{-1}\Psib(0)$ and $D^{-1}\tilde\Psib(0)$. By convexity of $\mathcal{D}$, and since $\Psib(0), \tilde \Psib(0) >0$,
\begin{equation}\label{rdef}
r \doteq  \inf_{\ubd \in \Lmc} \operatorname{dist}(\ubd, \partial \mathcal{D}) >0,
\end{equation}
where $\partial \mathcal{D}$ denotes the boundary of $\mathcal{D}$ and $\operatorname{dist}$ denotes Euclidean distance of a point from this set. Also define the processes
\begin{align*}
M(t) \doteq \frac{\vbd'}{\|\vbd\|}\Bbd^{(i)}(t), \
\Mbd^{\perp}(t) \doteq \left(I - \frac{\vbd \vbd'}{\|\vbd\|^2}\right)\Bbd^{(i)}(t) = \left(I - \frac{\vbd \vbd'}{\|\vbd\|^2}\right)\tilde\Bbd^{(i)}(t), \ t \ge 0,
\end{align*}
where the last equality can be verified from \eqref{mirror}. Observe that $M$ is a standard Brownian motion and moreover, $M$ and $\Mbd^{\perp}$ are independent Gaussian processes. From a geometric point of view, $2M$ denotes the component of $\tilde \Bbd^{(i)} - \Bbd^{(i)}$ along the vector $\vbd$ and $\Mbd^{\perp}$ denotes the `synchronously' coupled projections of $\tilde \Bbd^{(i)}$ and $\Bbd^{(i)}$ along the hyperplane perpendicular to $\vbd$. Define the stopping time
$$
\tau^* \doteq \inf\{t \ge 0: M(t) = -r/4 \text{ or } \|\Mbd^{\perp}(t)\| = r/4\}.
$$
Consider any $t \le \sigma \wedge \tau^* \wedge \frac{r}{8\|D^{-1}\bbd\|+1}$. Note that, if $M(t) \ge 0$, then, using $t \le \sigma$, 
$$\ubd \doteq D^{-1}\Psib(0) + M(t) \frac{\vbd}{\|\vbd\|} =
\left(1 - \frac{\vbd'\Bbd^{(i)}(t)}{\|\vbd\|^2}\right) D^{-1}\Psib(0) + \frac{\vbd'\Bbd^{(i)}(t)}{\|\vbd\|^2} D^{-1}\tilde \Psib(0)
 \in \Lmc.$$
  Furthermore,
$$
\|D^{-1}\Psib(t) - \ubd\| = \|\Mbd^{\perp}(t) + D^{-1}\bbd t\| \le \|\Mbd^{\perp}(t)\| + \|D^{-1}\bbd \|t \le \frac{3r}{8}.
$$
Hence, by definition of $r$, $D^{-1}\Psib(t) \notin \partial \mathcal{D}$. If, on the other hand, $M(t) <0$, then recalling $t \le \tau^*$,
\begin{align*}
\|D^{-1}\Psib(t) - D^{-1}\Psib(0)\| &= \|\Bbd^{(i)}(t) + D^{-1}\bbd t\|=
\|M(t) \frac{\vbd}{\|\vbd\|} + \Mbd^{\perp}(t) + D^{-1}\bbd t\|\\
& \le |M(t)| + \|\Mbd^{\perp}(t)\| + \|D^{-1}\bbd \|t \le \frac{5r}{8},
\end{align*}
again implying $D^{-1}\Psib(t) \notin \partial \mathcal{D}$. Hence, we conclude that, on the event $\{\sigma \le \tau^* \wedge \frac{r}{8\|D^{-1}\bbd\|+1}\}$, $D^{-1}\Psib(t) \notin \partial \mathcal{D}$ for all $t \le \sigma$. A similar argument gives $D^{-1}\tilde\Psib(t) \notin \partial \mathcal{D}$ for all $t \le \sigma$ on the same event.

Since $\gbd \in \cld$, the sequence $\{g_i\}$ is bounded. Let $g^l \in [1,\infty)$ be such that
$$g_j \ge -g^l \mbox{ for all } j \in \NN_0.$$
Fix any
\begin{equation}\label{eq:506}
	\underline{\delta}  \in (0, z_{i+1}/(2 + 4g^l)).\end{equation}
For $s \in (0, \underline{\delta}]$ and $\delta_1, \delta_2 \in (0, \underline{\delta})$, define the following events in $\clf$:
\begin{align}\label{evdef}
\cle^1(s) &\doteq \left\{ \inf_{j\ge i+1}\inf_{0 \le t \le s} (y_j +B_j(r)) > \frac{y_i + y_{i+1} + \delta_2}{2} + g^l \underline{\delta}\right\},\notag\\
\cle^2(s) & \doteq \left\{\inf_{t \le s \wedge \frac{r}{8\|D^{-1}\bbd\|+1}} M(t) > -r/4, \ \sup_{t \le s \wedge \frac{r}{8\|D^{-1}\bbd\|+1}} M(t) \ge \|\vbd\|/2, \ \sup_{t \le s \wedge \frac{r}{8\|D^{-1}\bbd\|+1}} \|\Mbd^{\perp}(t)\| < r/4\right\}.\notag\\
\end{align}

Let $\cle(s) \doteq \cle^1(s) \cap \cle^2(s)$. For notational convenience, we suppress the dependence of $\cle^1(s), \cle^2(s),\cle(s)$ on $\delta_1,\delta_2$.

We claim that $\cle(s) \subseteq \{\tau_c \le s\}$. To see this, observe that on the event $\cle^1(s)$, the ordered Atlas particles $X_{i+1}$ and $\tilde X_{i+1}$ stay above the level $\frac{y_i + y_{i+1} + \delta_2}{2}$ by time $s$. Moreover, on $\cle^2(s)$, $\{\sigma \le \tau^* \wedge \frac{r}{8\|D^{-1}\bbd\|+1} \wedge s\}$ which, by the previous discussion, implies $D^{-1}\Psib(t) \notin \partial \mathcal{D}$ and $D^{-1}\tilde\Psib(t) \notin \partial \mathcal{D}$ for all $t \le \sigma$, that is, none of the co-ordinates of $\Psib(\cdot)$ or $\tilde \Psib(\cdot)$ hit zero by time $\sigma$. Hence, for all $t \le \sigma$, $\Psib(t) = \left(Z_1(t), \dots, Z_i(t), \frac{y_i + y_{i+1} + \delta_2}{2} - X_{i}(t)\right)'$ and $\tilde\Psib(t) = \left(\tilde Z_1(t), \dots, \tilde Z_i(t), \frac{y_i + y_{i+1} + \delta_2}{2} - \tilde X_{i}(t)\right)'$. Further, by the mirror coupling dynamics, $\Psib(\sigma) = \tilde\Psib(\sigma)$ and thus, under $\cle(s)$,
$$
\left(Z_1(t), \dots, Z_i(t), \frac{y_i + y_{i+1} + \delta_2}{2} - X_{i}(t)\right) = \left(\tilde Z_1(t), \dots, \tilde Z_i(t), \frac{y_i + y_{i+1} + \delta_2}{2} - \tilde X_{i}(t)\right)
$$
for all $t \ge \sigma$. As $\sigma \le s$ under $\cle(s)$, we conclude that the above equality holds for all $t \ge s$. Finally, as $X_i$ and $X_{i+1}$ (also, $\tilde X_i$ and $\tilde X_{i+1}$) do not meet by time $\sigma$, $X_j(t) = \tilde X_j(t)$ for all $j \ge i+1$, for all $t \le \sigma$, and hence for all $t \ge 0$. These observations imply $\tau_c \le s$ on $\cle(s)$.

The following lemma gives a key estimate on the probabilities of this event. Recall $t_0>0$ for which $\psi = T_{t_0}\psi_0$.

\begin{lemma}\label{lemcty}
	For any $\eta>0$, there is a $\delta_0 \in (0, \underline{\delta})$ and $t_1 \in (0,t_0 \wedge \underline{\delta})$ such that $\PP_{\zbd, \delta_1,\delta_2,i}(\tau_c> t_1) \le \PP(\cle(t_1)^c) \le \eta$ for all $\delta_1, \delta_2 \in (0, \delta_0)$.
\end{lemma}
\begin{proof}
	Note that the inequality $\PP_{\zbd, \delta_1,\delta_2,i}(\tau_c> t_1) \le \PP(\cle(t_1)^c)$ is immediate from the discussion above the lemma.
	Now fix $\eta \in (0,1)$. Constants appearing in this proof may depend on $\zbd$ and this dependence is not noted explicitly. By a union bound and properties of Brownian motion we see that for any $\delta_2 \in (0, \underline{\delta})$ and $s \in (0,1)$,
	\begin{equation}\label{eq:eq508}
	\PP((\cle^1(s))^c) \le 2 \sum_{j=i+1}^{\infty} \bar \Phi\left(\frac{y_j - y_{i+1} + \frac{1}{2}(z_{i+1}-\delta_2) - g^l \underline{\delta}}{\sqrt{s}}\right),	
	\end{equation}
	where for $u \in \RR$,
	\begin{equation}\label{eq:gauscdf}
		\bar \Phi(u) = \frac{1}{\sqrt{2\pi}} \int_{u}^{\infty} e^{-v^2/2} dv. \end{equation}
	Note that, for $u \ge 0$, $\bar \Phi(u) \le \sqrt{2} e^{-u^2/4}$.
By our condition on $\underline{\delta}$ in \eqref{eq:506} and using $\delta_2 \in (0, \underline{\delta})$,
$$\frac{1}{2}(z_{i+1}-\delta_2) - g^l \underline{\delta} \ge \frac{1}{2} (z_{i+1} - (2g^l+1)\underline{\delta}) \doteq c_1 >0.$$
Thus, from \eqref{eq:eq508} it follows that
\begin{align}\label{eq:A}
\PP((\cle^1(s))^c) &\le \sqrt{8} e^{-c_1^2/4s}\sum_{j=i+1}^{\infty} e^{-(y_j-y_{i+1})^2/4s}\\
&\le \sqrt{8} e^{-c_1^2/4s}\sum_{j=i+1}^{\infty} e^{-(y_j-y_{i+1})^2/4}
\le \sqrt{8} e^{-c_1^2/4s}   e^{y_{i+1}^2/4}\sum_{j=i+1}^{\infty}e^{-y_j^2/8},
\end{align}
where the first inequality uses $(a+b)^2 \ge a^2+b^2$ for $a,b\ge 0$, the second uses the fact that $s \in (0,1)$, and the third uses the inequality $(a-b)^2 \ge a^2/2 - b^2$ for $a,b\in \RR$. Thus, choosing $t' \in (0,1)$ such that
$$\sqrt{8} e^{-c_1^2/4t'}   e^{y_{i+1}^2/4}\sum_{j=i+1}^{\infty}e^{-y_j^2/8} \le \eta/4,$$
(here we use $\ybd \in \clu$), we obtain
\begin{equation}\label{ebd1}
\PP((\cle^1(s))^c) \le \eta/4 \text{ for all } \delta_2 \in (0, \underline{\delta}) \text{ and all } s \le t'.
\end{equation}
Without loss of generality we assume that $t' < t_0 \wedge \underline{\delta}$. Next, note that, by the independence of $M$ and $\Mbd^{\perp}$, for any $s \in (0, \frac{r}{8\|D^{-1}\bbd\|+1}]$,
\begin{equation}\label{e20}
\PP((\cle^2(s))^c)
 = 1 - \PP\left(\inf_{t \le s } M(t) > -r/4, \ \sup_{t \le s} M(t) \ge \|\vbd\|/2\right)\PP\left(\sup_{t \le s} \|\Mbd^{\perp}(t)\| < r/4\right).
\end{equation}
Recall $r$ defined in \eqref{rdef}. Writing $r = r(\delta_1,\delta_2)$ to highlight its dependence on $\delta_1, \delta_2$, it follows from \eqref{rdef} and the explicit form of $\Psib(0)$ and $\tilde\Psib(0)$ that there exists $\delta' \in (0, \underline{\delta})$ and $c_2>0$ such that 
$$
\inf_{\delta_1, \delta_2 \in (0, \delta')} r(\delta_1,\delta_2) \doteq c_2>0.
$$
Hence, we obtain $t_1 \in (0, t' \wedge \frac{r}{8\|D^{-1}\bbd\|+1})$ such that
\begin{equation}\label{e21}
\PP\left(\inf_{t \le t_1} M(t) > -r/4\right) \ge 1- \eta/4, \ \ \PP\left(\sup_{t \le t_1} \|\Mbd^{\perp}(t)\| < r/4\right) \ge 1-\eta/4, \text{ for all } \delta_1, \delta_2 \in (0, \delta').
\end{equation}
Recall $\vbd = D^{-1}(\tilde\Psib(0) - \Psib(0))$. As $\tilde\Psib(0) - \Psib(0) = (0,\dots, 0 , \delta_1, - \delta_2)$ and $D^{-1}$ depends only on $i$ and not on $\delta_1, \delta_2$, we can obtain $c_3>0$ depending only on $i$ such that
$$
\|\vbd\| \le c_3 \sqrt{\delta_1^2 + \delta_2^2} \text{ for all } \delta_1, \delta_2>0.
$$
Hence, we obtain $\delta_0 \in (0, \delta')$ depending on $t_1$ such that
\begin{equation}\label{e22}
\PP\left(\sup_{t \le t_1} M(t) \ge \|\vbd\|/2\right) \ge1- \eta/4 \text{ for all } \delta_1, \delta_2 \in (0, \delta_0).
\end{equation}
From \eqref{e21} and \eqref{e22}, for all $\delta_1, \delta_2 \in (0, \delta_0)$,
\begin{equation}\label{e23}
\PP\left(\inf_{t \le t_1} M(t) > -r/4, \ \sup_{t \le t_1} M(t) \ge \|\vbd\|/2\right) \ge \PP\left(\inf_{t \le t_1} M(t) > -r/4\right) - \PP\left(\sup_{t \le t_1} M(t) < \|\vbd\|/2\right) \ge 1-\eta/2.
\end{equation}
Using \eqref{e21} and \eqref{e23} in \eqref{e20}, we obtain for all $\delta_1, \delta_2 \in (0, \delta_0)$,
\begin{equation}\label{ebd2}
\PP((\cle^2(t_1))^c) \le 1- (1-\eta/2)(1-\eta/4) \le 3\eta/4.
\end{equation}
From \eqref{ebd1} and \eqref{ebd2}, we conclude
$$
\PP((\cle(t_1))^c) \le \PP((\cle^1(t_1))^c)  + \PP((\cle^2(t_1))^c) \le \eta \text{ for all } \delta_1, \delta_2 \in (0, \delta_0),
$$
which proves the lemma.
\end{proof}

{\bf Proof of Proposition \ref{prop:prop2}.}
In order to prove the proposition it suffices to show the right continuity of $(\delta_1,\delta_2) \mapsto \psi(\zbd + \delta_1 \ebd_i - \delta_2\ebd_{i+1})$ at $(\delta_1,\delta_2)=(0,0)$. 
Recall that $\psi = T_{t_0}\psi_0$. 
Let $\veps>0$ be arbitrary and let $\eta \doteq \veps/ (2\|\psi_0\|_{\infty})$.
Let $\underline{\delta}$ be as in \eqref{eq:506} and, for this chosen $\eta$, let $\delta_0 \in (0, \underline{\delta})$ and $t_1 \in (0,t_0 \wedge \underline{\delta})$ be as in Lemma \ref{lemcty}.
For $\delta_1>0$, $\delta_2 \in (0, z_{i+1})$, let $\PP_{\zbd, \delta_1,\delta_2,i}$ be as in \eqref{eq:coup} and let $\EE_{\zbd, \delta_1,\delta_2,i}$ be the corresponding expectation operator. For any $\delta_1, \delta_2 \in (0, \delta_0)$,
\begin{align*}
|\psi(\zbd + \delta_1 \ebd_i - \delta_2\ebd_{i+1})- \psi(\zbd)|	&=
|\EE_{\zbd^{\delta_1,\delta_2,i}} \psi_0(\Zbd(t_0)) - \EE_{\zbd} \psi_0(\Zbd(t_0))|\\
&\le \EE_{\zbd,\delta_1,\delta_2,i} \left[|\psi_0(\Zbd^{(1)}(t_0)) - \psi_0(\Zbd^{(2)}(t_0))|1_{\{\tau_c > t_1\}}\right]\\
&\le 2\|\psi_0\|_{\infty} \PP_{\zbd,\delta_1,\delta_2,i} (\tau_c>t_1)\\
&\le  2\|\psi_0\|_{\infty} \PP(\cle(t_1)^c) \le 2\|\psi_0\|_{\infty}\eta = \veps,
\end{align*}
where the fourth inequality uses Lemma \ref{lemcty}.
Since $\veps>0$ is arbitrary, the result follows. \hfill \qed \\

{\bf Proof of Proposition \ref{prop:prop1}.}
Let $\zbd$ and $\delta_1,\delta_2$ be as in the statement of the proposition and let $\PP_{\zbd, \delta_1,\delta_2,i}$ be as in \eqref{eq:coup}.
Recall the event $\cle^2(s)$ from \eqref{evdef} (defined for any $s>0$) and consider the following modification of $\cle^1(s)$:
\begin{align*}
\tilde\cle^1(s) \doteq \left\{ \inf_{j\ge i+1}\inf_{0 \le t \le s} (y_j +B_j(t) - t g^l) > \frac{y_i + y_{i+1} + \delta_2}{2} \right\}, \ s >0.
\end{align*}
Let $\tilde \cle(s) \doteq \tilde\cle^1(s) \cap \cle^2(s)$. From the definition of the coupling $\PP_{\zbd, \delta_1,\delta_2,i}$ it is easily seen by an argument similar to that given before Lemma \ref{lemcty} that
$$\PP_{\zbd, \delta_1,\delta_2,i}(\Zbd^{(1)}(s) = \Zbd^{(2)}(s)) \ge \PP(\tilde \cle(s)).$$

Note that, as $\tilde\cle^1(s)$ is given in terms of  Brownian motions $\{B_j\}_{j \ge i+1}$ and $\cle^2(s)$ is defined in terms of $\{B_j\}_{j \le i}$, $\tilde\cle^1(s)$ and $\cle^2(s)$ are independent. Thus, for any $s>0$, $\PP(\tilde \cle(s)) = \PP(\tilde \cle^1(s))\PP(\cle^2(s))$ and it suffices to show that each term in the product is positive.

As $\ybd \in \mathcal{U}$, $y_j \rightarrow \infty$ as $j \rightarrow \infty$. Hence, we can obtain $j_0 \ge i+2$ depending on $s,i$ such that $y_j \ge 2sg^l + 2(y_i + y_{i+1} + \delta_2)$ for all $j \ge j_0$ and $\sum_{j=j_0}^{\infty} \bar \Phi\left(\frac{y_j}{4\sqrt{s}}\right) \le 1/4$. Thus,
\begin{align}\label{pos1}
\PP\left( \inf_{j\ge j_0}\inf_{0 \le t \le s} (y_j +B_j(t) - t g^l) > \frac{y_i + y_{i+1} + \delta_2}{2} \right) &\ge \PP\left( \inf_{j\ge j_0}\inf_{0 \le t \le s} (\frac{y_j}{2} +B_j(t)) > \frac{y_i + y_{i+1} + \delta_2}{2} \right)\notag\\
& \ge 1 - 2 \sum_{j=j_0}^{\infty} \bar \Phi\left(\frac{y_j}{4\sqrt{s}}\right) \ge 1/2.
\end{align}
Moreover, from standard Brownian motion estimates,
\begin{align}\label{pos2}
\PP\left( \inf_{i+1 \le j < j_0}\inf_{0 \le t \le s} (y_j +B_j(t) - t g^l) > \frac{y_i + y_{i+1} + \delta_2}{2} \right) >0.
\end{align}
From the independence of the events considered in \eqref{pos1} and \eqref{pos2}, we conclude that $\PP(\tilde \cle^1(s))>0$.
Finally, from standard Brownian motion estimates and the explicit form of $M(\cdot)$ and $\Mbd^{\perp}$ in terms of $\Bbd^{(i)}(\cdot)$,
\begin{multline*}
\PP\left(\inf_{t \le s \wedge \frac{r}{8\|D^{-1}\bbd\|+1}} M(t) \ge -r/4, \ \sup_{t \le s \wedge \frac{r}{8\|D^{-1}\bbd\|+1}} M(t) \ge \|\vbd\|/2, \ \sup_{t \le s \wedge \frac{r}{8\|D^{-1}\bbd\|+1}} \|\Mbd^{\perp}(t)\| \le r/4\right)\\
= \PP\left(\inf_{t \le s \wedge \frac{r}{8\|D^{-1}\bbd\|+1}} M(t) \ge -r/4, \ \sup_{t \le s \wedge \frac{r}{8\|D^{-1}\bbd\|+1}} M(t) \ge \|\vbd\|/2\right) \PP\left(\sup_{t \le s\wedge \frac{r}{8\|D^{-1}\bbd\|+1}} \|\Mbd^{\perp}(t)\| \le r/4\right) >0.
\end{multline*}
 The result follows. \hfill \qed

\section{Proof of Theorem \ref{thm_main2}}
\label{sec:mainsecthm}
Recall $\cld_1$ from \eqref{cldzdefn}. Observe that for $\gbd \in \cld_1$, $\bar g_{N_{j+1}} < \bar g_{N_j}$ for all $j \ge 1$. Moreover, as $\gbd \in \cld$, $\inf_{n \in \NN} \bar g_n \ge \inf_{n \in \NN} g_n > - \infty$. Thus, $\bar g_{\infty} \doteq \lim_{j \rightarrow \infty} \bar g_{N_j}$ exists, is finite, and $\lim_{j \rightarrow \infty} \bar g_{N_j} = \inf_{n \in \NN} \bar g_n$. As adding the same drift $-\bar g_{\infty}dt$ to each ordered particle in \eqref{order_SDE} keeps the gaps unchanged, we can assume without loss of generality that $\inf_{n \in \NN} g_n=\bar g_{\infty}=0$. In particular, $\bar g_n>0$ for all $n \in \mathbb{N}$.

Fix $\gbd \in \cld_1$ and let $\pi = \otimes_{i=1}^{\infty} \pi_i$ be as in the statement of
 Theorem \ref{thm_main2}. The assumption \eqref{eq:satinteg} on 
 $\pi$ will be taken to hold throughout the section. {\SB From Theorem \ref{sarthm} we can construct a filtered probability
space $(\Om, \clf, \PP, \{\clf_t\})$ equipped with mutually independent real $\{\clf_t\}_{t\ge 0}$-Brownian motions $\{B^*_i\}_{i \in \NN_0}$ and continuous processes $\{Y_{(i)}, i \in \NN_0\}$ that solve the SDE \eqref{order_SDE}, where $\{L^*_i\}_{i \in \NN_0}$ are as introduced below \eqref{order_SDE}, such that with $\{Z_i\}_{i\in \NN}$ defined as in \eqref{gapdef}, the process
$\Zbd = (Z_1, Z_2, \ldots)'$ has the  distribution $\PP^{\gbd}_{\pi}$. Furthermore, without loss of generality, 
we can assume that $Y_{(0)}(0)=0$. We will write $\PP$ and $\EE$ respectively for the probability and expectation under the law of this $\RR^{\infty}$-valued process.}
{\SB \subsection{Proof overview}\label{sec:pfov2}
First we give an overview of the approach. We will use moment generating functions (m.g.f) to identify the marginals of $\pi$; so the first step is to establish finiteness of the m.g.f. of any fixed gap in a positive interval around zero. This is achieved in Lemma \ref{lem:finmgf} by using comparison techniques between the gap processes of infinite and finite versions of the model, the latter having a unique invariant distribution that is a product of Exponential distributions. Lemmas \ref{lem:sqloctim} and \ref{lemunifinteg} together establish the uniform integrability of $\{\frac{1}{\eps} \int_0^1 1_{\{0 \le Z_i(s) \le \eps\}} ds, \; 
    \eps \in (0, 1/2)\}$ for any $i \in \NN$, which is later used in showing the existence of $\lim_{\eps \downarrow 0}\frac{1}{\eps}\pi_i[0,\eps]$ and to identify this as $\EE(L^*_i(1))$ ($L^*_i$ being the local time at zero of the $i$-th gap in \eqref{order_SDE}). Lemma \ref{lem:lociden}, which is key to the proof of Theorem \ref{thm_main2}, gives an explicit representation for the expectation of the integral of a function of the $i$-th gap process against the $j$-th local time process for $i \neq j$. The aforementioned uniform integrability is crucially used here. For any $i \in \NN$, the m.g.f. of the $i$-th gap at time $1$ is then identified by an application of It\^{o}'s formula to exponential functions of the gap and using the representation in Lemma \ref{lem:lociden} to evaluate the local time terms. The obtained m.g.f. corresponds to that of an exponential random variable. The associated rates are then shown to agree with that of $\pi^{\gbd}_a$ for some $a\ge -2 \inf_{n \in \NN} \bar g_n$ via a recursive relation resulting from taking expectations in \eqref{eq:zi}. The representation \eqref{arep} is obtained as a by-product of our computations.}
\subsection{Preliminary results}
We begin  with some preliminary results.
\begin{lemma}
\label{lem:finmgf}
 For any $i \in \NN$ and $\lambda < 2\sum_{k=0}^{i-1}g_k$, we have
 $\int_{\RR_+} e^{\lambda z} \pi_i(dz) <\infty$.   
\end{lemma}

\begin{proof}
    Recall the sequence $\{N_j\}$ associated with $\gbd \in \cld_1$. Fix any $d \in \mathbb{N}$ such that $N_d>i$ and consider the $N_d$ dimensional $(\gbd,\ybd)$-Atlas model defined by replacing $\infty$ with $N_d-1$ in equation 
    \eqref{eq:unrank}. This model has been studied extensively in previous works (see e.g. \cite{AS, AS2}) and it is well known that, since $\gbd \in \cld_1$, the associated gap sequence $\{Z_j\}_{j=1}^{N_d-1}$ defined by 
    \eqref{gapdef}, where the processes $Y_{(j)}$ are defined by 
    \eqref{order_SDE}, for $j = 0, 1, \ldots , N_d-1$, has a unique stationary distribution $\tilde \pi^{(N_d-1)} \doteq
    \otimes_{l=1}^{N_d-1} \mbox{Exp}(2l(\bar g_l - \bar g_{N_d}))$
    (see \cite[Proposition 2.2(4)]{AS}).
    Using monotonocity and comparison estimates for finite and infinite Atlas models (cf. \cite[Corollary 3.14]{AS} and \cite[equations (58)-(60)]{AS2}) it now follows that the probability measure $\pi|_{N_d-1}$ on $\RR_+^{N_d-1}$ given as the 
    first $N_d-1$ marginal distribution of $\pi$ satisfies 
    $\pi|_{N_d-1} \lest \tilde \pi^{(N_d-1)}$.
    In particular, $\pi_i \lest \mbox{Exp}(2i(\bar g_i - \bar g_{N_d}))$ for all $N_d>i$.
    Since $\bar g_{N_d} \to 0$ as $d\to \infty$, we can find a $d \in \NN$, with $N_d>i$, such that $\lambda < 2i(\bar g_i - \bar g_{N_d})$.
    Then 
    $$\int_{\RR_+} e^{\lambda z} \pi_i(dz) \le
    \int_{\RR_+} e^{\lambda z} \mbox{Exp}(2i(\bar g_i - \bar g_{N_d}))(dz) <\infty$$
    which completes the proof.
\end{proof}

The next three lemmas concern the collection $\{Z_i, L^*_i\}$ described above.
{\SB We remind the reader that in these lemmas the probability measure $\mathbb{P}$ and the expectation $\mathbb{E}$ correspond to $\PP^{\gbd}_{\pi}$ and $\EE^{\gbd}_{\pi}$ respectively, where $\gbd$ and $\pi$ are as fixed at the beginning of the section.}
\begin{lemma}
\label{lem:sqloctim}
    For every $i \in \NN$, $\EE(L^*_i(1))^2<\infty$.
\end{lemma}
\begin{proof}
    From \eqref{order_SDE} it follows that, for $i \in \NN$, and $0\le t \le 1$,
    \begin{equation}\label{eq:zi}
        Z_i(t) = Z_i(0) + h_it + W_i^*(t) - \frac{1}{2}L^*_{i-1}(t)
        - \frac{1}{2}L^*_{i+1}(t) + L^*_i(t),
    \end{equation}
    where for $i\in \NN$, $h_i = g_i - g_{i-1}$ and $W_i^* = B_i^* - B_{i-1}^*$. 
    This says that
    $$L^*_{i+1}(1) = 2(Z_i(0)-Z_i(1) +h_i + W_i^*(1) + L^*_i(1)) -L^*_{i-1}(1).$$
From this and Lemma \ref{lem:finmgf} it follows that if for some $i\in \NN$
$\EE(L^*_i(1))^2 <\infty$, then
$\EE(L^*_{i+1}(1))^2 <\infty$ as well. Thus it suffices to show that
$\EE(L^*_{1}(1))^2 <\infty$.
From \eqref{order_SDE}, and recalling that $Y_{(0)}(0)=0$, we see that
$$Y_{(0)}(1) = g_0 + B^*_1(1) - \frac{1}{2} L^*_1(1)$$
which says that
$$L^*_1(1) = 2(-Y_{(0)}(1) + g_0 + B^*_1(1)).$$
Thus to prove the lemma it suffices to show that
\begin{equation}\label{eq:dessqint}
    \EE\left( \inf_{j\in \NN_0} Y_j(1)\right)^2<\infty,
\end{equation}
where $\{Y_j\}$ solve the system of equations in \eqref{eq:unrank} with $Y_0(0)=0$
and the vector
$$(Y_1(0), Y_2(0)- Y_1(0), Y_3(0)- Y_2(0), \ldots )$$
distributed as $\pi$.

{\SB Note that for $x > g_0$,
\begin{equation}\label{ygreat}
\PP(\inf_{j\in \NN_0}  Y_j(1) > x) \le \PP(g_0 + B^*_1(1) > x) \le\sqrt{2}e^{-(x-g_0)^2/4}.
\end{equation}}

 Next, as $\gbd \in \cld$, there exists $g^l>0$ such that $g_i \ge -g^l$ for all $i \in \mathbb{N}_0$. Thus, we have, for $x\ge 0$,
\begin{align*}
 \PP(\inf_{j\in \NN_0}  Y_j(1) < -x)
 &\le  \PP(\inf_{j\in \NN_0} (Y_j(0) + W_j(1))
  < g^l -x)\\
  &\le \sum_{j=0}^{\infty}\PP( W_j(1)
  \le g^l -x-Y_j(0)) 
  =  \sum_{j=0}^{\infty} \EE \bar \Phi(x+ Y_j(0)-g^l),
\end{align*}
where $\bar \Phi(z)$ was defined in \eqref{eq:gauscdf}.
Thus, for $x \ge g^l$,
\begin{align*}
 \PP(\inf_{j\in \NN_0} Y_j(1) \le -x)
 &\le \sqrt{2} \sum_{j=0}^{\infty} \EE e^{-(x+ Y_j(0) - g^l)^2/4}
 \le \sqrt{2}e^{-(x-g^l)^2/4} \sum_{j=0}^{\infty} 
 \EE e^{-(Y_j(0))^2/4}.
\end{align*}
The desired square integrability in \eqref{eq:dessqint} is now immediate {\SB from \eqref{ygreat} and the above} on observing that
$$\sum_{j=0}^{\infty} 
 \EE e^{-(Y_j(0))^2/4}
 = 1 + \int_{\RR_+^{\infty}} \left(\sum_{j=1}^{\infty}
        e^{-\frac{1}{4} (\sum_{l=1}^j z_l)^2} \right) \pi(d\zbd) <
        \infty$$
        by our assumption.
\end{proof}
The next lemma will allow us to interchange expectations and limits as $\eps \to 0$.
\begin{lemma}\label{lemunifinteg}
    The family 
    $\{\frac{1}{\eps} \int_0^1 1_{\{0 \le Z_i(s) \le \eps\}} ds, \; 
    \eps \in (0, 1/2)\}$ is uniformly integrable.
\end{lemma}
\begin{proof}
    For $\eps \in (0, 1/2)$, define $\psi_{\eps}: \RR \to \RR$ as
    \begin{equation}
        \psi_{\eps}(z) \doteq
        \begin{cases}
            \frac{z^2}{2} & \mbox{ if } 0 \le z \le \eps\\
            \frac{\eps^2}{2} + (z-\eps)\eps & \mbox{ if } z > \eps.
        \end{cases}
    \end{equation}
    Note that $\psi_{\eps}$ is a $C^1$ function with an absolutely continuous derivative. It then follows from It\^{o}'s formula
    (see \cite[Problem 3.7.3]{KSbook}) applied to $Z_i$ given by 
    \eqref{eq:zi}, that
    \begin{align}
      \psi_{\eps}(Z_i(1)) &=  \psi_{\eps}(Z_i(0)) 
      + \int_0^1 \psi'_{\eps}(Z_i(s)) h_i ds +
      \int_0^1 \psi'_{\eps}(Z_i(s)) dW^*_i(s)\nonumber\\
      &\quad-\frac{1}{2} \int_0^1 \psi'_{\eps}(Z_i(s)) dL^*_{i+1}(s)
      -\frac{1}{2} \int_0^1 \psi'_{\eps}(Z_i(s)) dL^*_{i-1}(s)\nonumber\\
      &\quad+\int_0^1 \psi'_{\eps}(Z_i(s)) dL^*_{i}(s)
      + \int_0^1 \psi''_{\eps}(Z_i(s)) ds.\label{eq:1001}
\end{align}
Note that, for all $z \in \RR_+$,
$$0\le \psi_{\eps}(z) \le z\eps, \;\; 0\le \psi'_{\eps}(z) \le \eps, \;\; \psi'_{\eps}(0)=0.$$
Also,
$$\psi''_{\eps}(z) = \begin{cases}
    1 & \mbox{ if } 0 \le z < \eps\\
    0 & \mbox{ if } z > \eps.
\end{cases}
$$
Combining these, and dividing by $\eps$ in \eqref{eq:1001}, we have
\begin{align*}
\frac{1}{\eps} \int_0^1 1_{\{0 \le Z_i(s) \le \eps\}} ds
&\le Z_i(1) - \frac{1}{\eps} \int_0^1 \psi'_{\eps}(Z_i(s)) dW^*_i(s) + |h_i| + \frac{1}{2} L^*_{i-1}(1) + \frac{1}{2} L^*_{i+1}(1).
\end{align*}
The desired uniform integrability now follows from Lemmas \ref{lem:finmgf} and \ref{lem:sqloctim} and the observation that
$$\EE\left(\frac{1}{\eps} \int_0^1 \psi'_{\eps}(Z_i(s)) dW^*_i(s)\right)^2
\le 2.$$
\end{proof}
 The following lemma will be key to proving Theorem \ref{thm_main2}. It will be used to represent expectations of integrals of nonnegative measurable functions with respect to local time in terms of stationary integrals and the `density' of $\pi_i$ at zero for each $i$, as described in the lemma. We note that the product form structure of $\pi$ is crucially exploited here.
\begin{lemma}
	\label{lem:lociden}
For any $i\in \NN$, the limit $\nu_i \doteq \lim_{\eps \downarrow 0}
\frac{1}{\eps}\pi_i[0,\eps]$ exists and $\nu_i = \EE(L^*_i(1))$.
Furthermore, for any measurable $f: \RR_+ \to \RR_+$ and $i, j \in \NN$, $i\neq j$,
\begin{equation}\label{eq:854n}
    \EE \int_0^1 f(Z_i(s)) dL^*_j(s) = \nu_j \int_{\RR_+} f(z) \pi_i(dz).
\end{equation}

\end{lemma}
\begin{proof}
From results on local times of continuous semimartingales (see e.g. \cite[Corollary VI.1.9]{RY}) it follows that, for all $t \in [0,1]$ and $i \in \NN$,
$\frac{1}{\eps} \int_0^t 1_{\{0 \le Z_i(s) \le \eps\}} ds$ converges a.s. to $1/2$-times the semimartingale local time $\Lambda_i(t)$ of $Z_i$ at $0$ (as defined in \cite[VI.1.2]{RY}) as $\eps \downarrow 0$. Furthermore, one has (see \cite[Exercise VI.1.16 (3)]{RY}) that
\begin{align*}\Lambda_i(t) &= 2 \Big[h_i\int_0^t 1_{\{Z_i(s)=0\}} ds
- \frac{1}{2} \int_0^t 1_{\{Z_i(s)=0\}} d L^*_{i-1}(s)\\
&\quad- \frac{1}{2} \int_0^t 1_{\{Z_i(s)=0\}} d L^*_{i+1}(s)
+\int_0^t 1_{\{Z_i(s)=0\}} d L^*_{i}(s)\Big]\\
&= 2\int_0^t 1_{\{Z_i(s)=0\}} d L^*_{i}(s) = 2L^*_{i}(t),
\end{align*}
where the first equality on the last line follows from   the facts that $\int_0^t 1_{\{Z_i(s)=0\}} ds=0$ (which follows from $\Lambda_i(t) < \infty$), and that the Atlas model does not have triple collisions a.s. (see \cite[Theorem 5.1]{AS}). It then follows that, as $\eps \downarrow 0$,
\begin{equation}\label{eq:548n}\frac{1}{\eps} \int_0^t 1_{\{0 \le Z_i(s) \le \eps\}} ds 
\to L^*_{i}(t), \mbox{ a.s. for every } t \in [0,1] \mbox{ and } i \in \NN.\end{equation}
Combining this with Lemma \ref{lemunifinteg} and using the fact that
$\Zbd$ is a stationary process, we now have that, as $\eps \downarrow 0$,
\begin{equation} \label{eq:602n}\frac{1}{\eps}\pi_i[0,\eps]
=\EE\frac{1}{\eps} \int_0^1 1_{\{0 \le Z_i(s) \le \eps\}} ds \to 
\EE(L^*_{i}(1)), \mbox{ for all } i \in \NN.\end{equation}
This proves the first statement in the lemma.
In order to prove \eqref{eq:854n} it suffices to consider the case where $f$ is bounded (as the general case can be then recovered by monotone convergence theorem). In fact by appealing to the monotone class theorem (cf. \cite[Theorem I.8]{prott}) we can assume without loss of generality  that $f$ is a continuous and bounded function.
From \eqref{eq:548n} we can find $\Om_0 \in \clf$ such that $\PP(\Om_0)=1$ and for all $\om \in \Om_0$
\begin{equation}\label{eq:548nn}\frac{1}{\eps} \int_0^t 1_{\{0 \le Z_i(s,\om) \le \eps\}} ds 
\to L^*_{i}(t, \om), \mbox{ for every } t \in [0,1] \cap \QQ, \end{equation}
and $L^*_i(1,\om) <\infty$ for all $i \in \NN$.
In particular this says that, for every $\om \in \Om_0$ and $i\in \NN$, the collection of measures $\{\Pi^{\eps, \om}_i, \eps \in (0, 1/2)\}$ defined as
$$\Pi^{\eps, \om}_i[a,b] \doteq \frac{1}{\eps} \int_a^b 1_{\{0 \le Z_i(s,\om) \le \eps\}} ds, \; 0 \le a \le b \le 1$$
is relatively compact in the weak convergence topology. From this and using \eqref{eq:548nn} again we now see that, for every $\om \in \Om_0$ and $i\in \NN$, $\Pi^{\eps, \om}_i$ converges weakly to $\Pi^{\om}_i$ defined as
$$\Pi^{\om}_i[a,b] \doteq L^*_i(b, \om) - L^*_i(a, \om), \; 0 \le a \le b \le 1.$$
From the sample path continuity of $Z_i$ we can assume without loss of generality that for every $\om \in \Om_0$ and $i\in \NN$, $s \mapsto f(Z_i(s,\om))$ is a continuous  map.
From the above weak convergence it then follows that, for $\om \in \Om_0$ and $i,j\in \NN$,
\begin{align*}
	\int_0^1 f(Z_i(s,\om)) dL^*_j(s,\om) &= \int_0^1 f(Z_i(s,\om)) d\Pi^{\om}_j(s)\\
	&= \lim_{\eps \to 0}  \int_0^1 f(Z_i(s,\om)) d\Pi^{\eps, \om}_j(s)
	= \lim_{\eps \to 0}  \frac{1}{\eps}\int_0^1 f(Z_i(s,\om)) 1_{\{0 \le Z_j(s,\om) \le \eps\}} ds.
\end{align*}
Using Lemma \ref{lemunifinteg}  and the fact that $\Zbd$ is a stationary process with a product form stationary distribution, we now see that for $i,j\in \NN$, $i\neq  j$,
\begin{align*}
	\EE\int_0^1 f(Z_i(s)) dL^*_j(s) &= \lim_{\eps \to 0}  \frac{1}{\eps}\EE\int_0^1 f(Z_i(s)) 1_{\{0 \le Z_j(s) \le \eps\}} ds\\
&=\lim_{\eps \to 0}  \frac{1}{\eps}\pi_j[0,\eps] \int f(z) \pi_i(dz) = 
	 \nu_j \int f(z) \pi_i(dz).\end{align*}

\end{proof}

\subsection{Proof of Theorem \ref{thm_main2}}
We now complete the proof of the theorem.
Recalling Lemmas \ref{lem:finmgf} and \ref{lem:sqloctim}),  taking expectations in \eqref{eq:zi}, and using the identity $\nu_i = \EE(L^*_i(1))$ for $i \in \NN$, we see that, for all $i \in \NN$,
\begin{equation}\label{eq:nuheq}
	h_i +\nu_i -\frac{1}{2} \nu_{i+1} - \frac{1}{2} \nu_{i-1} = 0.
\end{equation}
Applying the above identity for $i=1$, and setting $a \doteq \nu_1-2g_0$, we have
$$
\nu_2 = 2\nu_1 + 2h_1 = 2(g_0+g_1) + 2(\nu_1-2g_0) = 2(g_0+g_1) + 2a.$$
Proceeding by induction, suppose that for some $k \ge 2$,
\begin{equation}
\nu_i = ia + 2(g_0+ \cdots + g_{i-1}), \mbox{ for all } 1 \le i \le k.	
\end{equation}
Then from \eqref{eq:nuheq},
\begin{align*}
\nu_{k+1} &= 2(\nu_k - \frac{1}{2} \nu_{k-1} + h_k)	\\
&= 2 ak + 4(g_0+ \cdots + g_{k-1}) - a(k-1) - 2(g_0+ \cdots + g_{k-2}) + 2(g_k- g_{k-1})\\
&= (k+1) a + 2(g_0+ \cdots + g_{k}).
\end{align*}
Thus it follows that, for every $k\ge 1$,
\begin{equation}\label{eq:nukrec}\nu_k = ka + 2(g_0+ \cdots + g_{k-1}) = k(a+ 2 \bar g_k).\end{equation}
Fix  $i \in \NN$ and $\lambda < \sum_{k=0}^{i-1}g_k$. Then by It\^{o}'s formula applied to $Z_i$ given by \eqref{eq:zi},
    \begin{align}\label{itoexp}
        e^{\la Z_i(1)} &= e^{\la Z_i(0)}  + \la h_i \int_0^1 e^{\la Z_i(s)} ds  + \la \int_0^1 e^{\la Z_i(s)} dW_i^*(s) - \frac{\la }{2} \int_0^1 e^{\la Z_i(s)} dL^*_{i-1}(s)\notag\\
        &\quad- \frac{\la }{2}\int_0^1 e^{\la Z_i(s)} dL^*_{i+1}(s) + \la \int_0^1 e^{\la Z_i(s)} dL^*_i(s)
		+ \la^2 \int_0^1 e^{\la Z_i(s)} ds.
    \end{align}
	Since $\lambda < \sum_{k=0}^{i-1}g_k$, from Lemma \ref{lem:finmgf},
	$\int_0^1 \EE e^{2\la Z_i(s)} ds = \int_{\RR_+} e^{2\la z} \pi_i(dz) <\infty$ and consequently the  stochastic integral in the above display has mean $0$. {\SB Moreover, note that
	$$
	\EE\int_0^1 e^{\la Z_i(s)} dL^*_i(s) = \EE\int_0^1 dL^*_i(s) = \EE(L^*_i(1)) = \nu_i.
	$$
	Thus taking expectations in \eqref{itoexp}} and using Lemma \ref{lem:lociden}, we have,
\begin{align*}
	\int_{\RR_+} e^{\la z} \pi_i(dz) &=\int_{\RR_+} e^{\la z} \pi_i(dz) + \la h_i \int_{\RR_+} e^{\la z} \pi_i(dz) - \frac{\la }{2} \nu_{i-1}\int_{\RR_+} e^{\la z} \pi_i(dz)\\
	&\quad - \frac{\la }{2} \nu_{i+1}\int_{\RR_+} e^{\la z} \pi_i(dz) + \la \nu_i +\la^2\int_{\RR_+} e^{\la z} \pi_i(dz).
\end{align*}
Rearranging terms, we have
$$
\la \nu_i = \left(-\la h_i +  \frac{\la }{2} \nu_{i-1} + \frac{\la }{2} \nu_{i+1} - \la^2\right)
\int_{\RR_+} e^{\la z} \pi_i(dz)  = (\la \nu_i - \la^2) \int_{\RR_+} e^{\la z} \pi_i(dz),$$
where the second equality in the above display follows from \eqref{eq:nuheq}.
Thus we have shown that for all $i \in \NN$ and $\lambda < \sum_{k=0}^{i-1}g_k$,
$$\int_{\RR_+} e^{\la z} \pi_i(dz) = \frac{\nu_i}{\nu_i-\la}.$$
Thus, since $\sum_{k=0}^{i-1}g_k = i\bar g_i>0$, by uniqueness of Laplace transforms, we must have that $\pi_i = \mbox{Exp}(\nu_i)$ for each $i \in \NN$. 
Finally note that, since by Lemma  \ref{lem:finmgf},
	$\int_{\RR_+} e^{\la z} \pi_1(dz) <\infty$ for all $\la < 2 g_0$, we must have that 
	$a = \nu_1-2g_0 \ge 0$.
Thus, in view of \eqref{eq:nukrec}, we have shown that, for some $a\ge 0$, $\pi_i = \mbox{Exp}(\nu_i) = \mbox{Exp}(i(2\bar g_i + a))$ for all $i \in \NN$ and so
$\pi = \pi_a^{\gbd}$ for some $a\ge 0$.

{\SB The assertion \eqref{arep} follows from \eqref{eq:nukrec} upon recalling that $\nu_k = \EE(L^*_k(1))$ for $k \in \NN$.}
\hfill \qed

\appendix
\section{Proof of Lemma \ref{lemerg}}
We will like to acknowledge the lecture notes of Sethuraman \cite{SethLN} which are used at several steps in the proof below.

Since $\gbd \in \cld$ will be fixed in the proof we suppress it from the notation. 
Let $\gamma \in \cli$. 
 
 Define
 $$\GG \doteq \{\eta - T_t \eta: \eta \in L^2(\gamma), \, t\ge 0\}.$$
 We claim that 
  $[(\mbox{span}(\GG))^{cl}]^\perp \subset \II_{\gamma}$, where $[(\mbox{span}(\GG))^{cl}]^\perp$ denotes the orthogonal complement of the closure of the linear subspace generated by  $\GG$ in  $L^2(\gamma)$.
 Indeed, if $\psi \in [(\mbox{span}(\GG))^{cl}]^{\perp}$, then for all $\eta \in L^2(\gamma)$ and $t\ge 0$,
 $\lan \psi, \eta - T_t\eta\ran =0$. Taking $\eta = \psi$, 
 \begin{equation}\label{eq:ap1036}
	 \lan \psi, \psi\ran = \lan \psi, T_t\psi\ran \mbox{ for all } t \ge 0.\end{equation}
 Using this, and the contraction property of $T_t$,  for all $t \ge 0$,
 \begin{align*}
	 0 &\le \lan T_t\psi - \psi, T_t\psi -\psi\ran = \lan T_t\psi , T_t\psi\ran + \lan \psi, \psi \ran - 2 \lan \psi, T_t \psi\ran\\
	 & = \lan T_t\psi , T_t\psi\ran + \lan \psi, \psi \ran - 2 \lan \psi,  \psi\ran
 =\lan T_t\psi , T_t\psi\ran - \lan \psi,  \psi\ran \le 0,
\end{align*}
where the first equality on the second line follows from \eqref{eq:ap1036}.
 This says that $ T_t\psi =\psi$ (as elements of $L^2(\gamma)$) for all $t\ge 0$ and so
 $\psi \in \II_{\gamma}$ and shows the claim $[(\mbox{span}(\GG))^{cl}]^{\perp} \subset \II_{\gamma}$. 
  Any $\psi\in L^2(\gamma)$ can be written as $\psi = \hat \psi_{\gamma} + \tilde \psi_{\gamma}$ where $\hat \psi_{\gamma} \in  [(\mbox{span}(\GG))^{cl}]^{\perp} \subset \II_{\gamma}$
 and $\tilde \psi_{\gamma} \in (\mbox{span}(\GG))^{cl}$.
 
 Next, for $\psi \in L^2(\gamma)$ and $t>0$,  define $A_t\psi \in L^2(\gamma)$ as
 $$A_t\psi \doteq \frac{1}{t} \int_0^t T_s \psi \,  ds.$$
 Then $A_t \psi = A_t \hat \psi_{\gamma} + A_t \tilde \psi_{\gamma}$.  By definition $A_t \hat \psi_{\gamma} = \hat\psi_{\gamma}$.
 Also, if $\phi \in \GG$ then for some $t_0\ge 0$ and $\eta \in L^2(\gamma)$
 $\phi = \eta - T_{t_0}\eta$.
 Then from the contraction property of $T_t$ it follows that
 $$\|A_t\phi\| \le \frac{2\|\eta\| t_0}{t} \to 0 \mbox{ as } t \to \infty.$$
 Similarly, if $\phi \in \mbox{span}(\GG)$, there is a $c(\phi) \in (0,\infty)$ such that
 $$\|A_t\phi\| \le \frac{c(\phi)}{t} \to 0 \mbox{ as } t \to \infty.$$
 Finally, let $\phi \in (\mbox{span}(\GG))^{cl}$. Then, given $\veps>0$, there is a $\phi^{\veps} \in \mbox{span}(\GG)$ such that $\|\phi - \phi^{\veps}\| \le \veps$. It follows, using again the contraction property, that for all $t>0$,
 $$\|A_t\phi\| \le \|A_t(\phi - \phi^{\veps})\| + \|A_t\phi^{\veps}\| \le \|\phi - \phi^{\veps}\| + \frac{c(\phi^{\veps})}{t} \le \veps + \frac{c(\phi^{\veps})}{t} .$$
 Thus
 $\limsup_{t\to \infty} \|A_t\phi\| \le \veps$. Since $\veps>0$ is arbitrary, we obtain 
 \begin{equation} \label{eq:ap1038}
	 \lim_{t\to \infty} \|A_t\phi\| = 0 \mbox{ for all } \phi \in (\mbox{span}(\GG))^{cl}.
\end{equation}
 From these observations we obtain, for any {\SB$\psi\in L^2(\gamma)$},
 \begin{equation}\label{eq:334}
	 \lim_{t \to \infty} A_t \psi = \lim_{t \to \infty} (A_t \hat \psi_{\gamma} +  A_t \tilde \psi_{\gamma}) = \lim_{t \to \infty} (\hat \psi_{\gamma} +  A_t \tilde \psi_{\gamma}) = \hat \psi_{\gamma}.
 \end{equation}
The above convergence in fact shows that $[(\mbox{span}(\GG))^{cl}]^{\perp} = \II_{\gamma}$  and consequently $\hat \psi_{\gamma}$ is the projection of $\psi$ on to  $\II_{\gamma}$.
To see this, recall that it was argued above that $[(\mbox{span}(\GG))^{cl}]^{\perp} \subset \II_{\gamma}$.
Now consider the reverse inclusion and
let  $ \varphi \in \II_{\gamma}$. Then we can write $\varphi = \varphi_1 +\varphi_2$ where $\varphi_1 \in (\mbox{span}(\GG))^{cl}$ and $\varphi_2 \in [(\mbox{span}(\GG))^{cl}]^{\perp} \subset \II_{\gamma}$. Thus, for $t\ge 0$,
$$\varphi = A_t \varphi = A_t \varphi_1 + A_t\varphi_2 = A_t \varphi_1 + \varphi_2.$$
As $t \to \infty$, we have from \eqref{eq:ap1038} that,  $A_t \varphi_1 \to 0$, which says that $\varphi = \varphi_2$. This proves the inclusion $\II_{\gamma} \subset [(\mbox{span}(\GG))^{cl}]^{\perp}$ and we have the claimed statement $[(\mbox{span}(\GG))^{cl}]^{\perp} = \II_{\gamma}$.

 Now we proceed to the proof of the statements in the lemma. We first consider the second statement in the lemma. Fix $\gamma \in \cli$. Suppose that for every bounded measurable map $\psi: \RR_+^{\infty} \to \RR$, $\hat{\psi}_{\gamma}$ is constant $\gamma$ a.s. 
 We will now show that this implies $\gamma \in \cli_e$. Suppose there is a $\veps \in (0,1)$ and $\gamma_1, \gamma_2 \in \cli$ such that
 $\gamma = \veps \gamma_1 + (1-\veps)\gamma_2$. Note that  since from \eqref{eq:334} $\hat{\psi}_{\gamma} = \lim_{t \to \infty} A_t \psi $ and $\gamma$ is invariant, we must have
 $\hat{\psi}_{\gamma} = \int_{\RR_+^{\infty}} \psi\, d\gamma$, and so from \eqref{eq:334} it follows that
 \begin{equation}
	 \int_{\RR_+^{\infty}} \left(A_t\psi - \int \psi d\gamma\right)^2 d\gamma = \|A_t\psi-\hat{\psi}_{\gamma}\|^2 \to 0 \mbox{ as } t \to \infty.
 \end{equation}
 Also, from definition,
 \begin{equation}
	 \limsup_{t\to \infty }\int_{\RR_+^{\infty}} \left(A_t\psi - \int \psi d\gamma\right)^2 d\gamma_1 \le
	  \limsup_{t\to \infty }\veps^{-1} \int_{\RR_+^{\infty}} \left(A_t\psi - \int \psi d\gamma\right)^2 d\gamma = 0.
 \end{equation}
 Thus $A_t\psi \to \int \psi d\gamma$ in $L^2(\gamma_1)$. 
 Also, since $\gamma_1 \in \cli$, $\int A_t \psi d \gamma_1 = \int \psi d\gamma_1$
 %
%
and consequently, 
 $\int \psi d\gamma_1 = \int \psi d\gamma$. Since $\psi$ is an arbitrary bounded measurable function, we must have $\gamma = \gamma_1$. This proves that $\gamma \in\cli_e$.
 We have thus shown the second statement in the lemma and in fact also shown that $\cli_{er} \subset \cli_e$.
 
 Finally we argue that $\cli_e \subset \cli_{er}$. Suppose $\gamma \in \cli_e$ and that $\gamma \not \in \cli_{er}$. Then  there is a $\psi \in L^2(\gamma)$ such that
 $\hat \psi_{\gamma}$ is not a.s. constant under $\gamma$. Thus there is a $c \in \RR$ such that, with $A = \{\hat \psi_{\gamma} >c\}$, $\gamma(A) \doteq \veps \in (0,1)$.
 Note that by definition 
 \begin{equation}\label{eq:1100}
	 T_t \hat \psi_{\gamma} = \hat \psi_{\gamma},\; \gamma  \mbox{ a.s. } \mbox{ for all } t \ge 0.
\end{equation} 
We refer to this property as $\hat \psi_{\gamma}$ is {\em harmonic} (with respect to the semigroup $\{T_t\}$).
 We claim that this implies that $1_A$ is harmonic as well, namely
 \begin{equation}\label{eq:114}
 	T_t 1_A = 1_A , \, \gamma \mbox{ a.s. } \mbox{ for all } t \ge 0.
 \end{equation}
 To see this, note that, from the definition of $T_t$, if $f \in L^2(\gamma)$ is harmonic, then
 \begin{equation}\label{eq:ap1046}
	 |f| = |T_t f | \le T_t|f|, \; \gamma \mbox{ a.s. } \end{equation}
 This,  together with the fact that $T_t$ is a contraction, says that
 $$\|f\| 
 \le \|T_t|f|\| \le \|f\|$$
 which in view of \eqref{eq:ap1046} shows that 
 $|f|$ is harmonic. From the linearity of $T_t$ it then follows that $f\vee 0 = \frac{1}{2} (f +|f|)$ is harmonic as well. This implies that if $f_1, f_2$ are harmonic then $f_1 \vee f_2$ and $f_1 \wedge f_2$ are harmonic as well. 
 Recalling tha $\hat \psi_{\gamma}$ is harmonic, we now have that $g_n \doteq \left(n\left(\hat\psi_{\gamma}-c\right)^+ \wedge 1\right)$ is harmonic for every $n \in \mathbb{N}$. The property in \eqref{eq:114} is now immediate from this on observing that $g_n \to 1_A$ a.s. and dominated convergence theorem. This proves the claim.
 
 Consider the probability measures $\gamma_1, \gamma_2$ on $(\RR_+^{\infty}, \clb(\RR_+^{\infty}))$ defined as
 $$\gamma_1(B) \doteq \veps^{-1}\gamma(B\cap A), \;\; \gamma_2(B) \doteq (1-\veps)^{-1}\gamma(B\cap A^c), \;\; B \in \clb(\RR_+^{\infty}).$$
 Using \eqref{eq:114} it is easily seen that $\gamma_1, \gamma_2 \in \cli$. 
 Indeed, if $B \in \clb(\RR_+^{\infty})$ and $t\ge 0$,
 \begin{align*}
	 \int_{\RR_+^{\infty}} T_t 1_B \, d\gamma_1 &=  \veps^{-1} \int_A T_t1_B \, d\gamma
	 =\veps^{-1} \int_A T_t1_{AB} \, d\gamma + \veps^{-1} \int_A T_t1_{A^cB} \, d\gamma\\
	 &= \veps^{-1} \int_A T_t1_{AB} \, d\gamma = \veps^{-1} \int_{\RR_+^{\infty}} T_t1_{AB} \, d\gamma
		%
	 =\veps^{-1}  \gamma (B\cap A) = \gamma_1(B),
  \end{align*}
  where the third and fourth equalities follow from \eqref{eq:114} and the fifth uses the invariance of $\gamma$. 
  This shows the invariance of $\gamma_1$ from which (together with the fact that $\gamma$ is invariant)
   the invariance of $\gamma_2$ follows immediately.
 
 Finally note that $\gamma_1\neq \gamma_2$, and by definition $\gamma = \veps \gamma_1 + (1-\veps)\gamma_2$. This contradicts the fact that $\gamma \in \cli_e$ and thus we must have $\gamma \in \cli_{er}$.  We have thus shown that $\cli_e \subset \cli_{er}$ which completes the proof. \hfill \qed
 
 \noindent {\bf Acknowledgements:}
Research supported in part by  the RTG award (DMS-2134107) from the NSF.  SB was supported in part by the NSF-CAREER award (DMS-2141621).
AB was supported in part by the NSF (DMS-2152577). {\SB We thank two anonymous referees whose valuable inputs significantly improved the article.}

{\noindent {\bf Data availability statement:}
This manuscript has no associated data.
}

\bibliographystyle{plain}
\bibliography{atlas_ref}

\vspace{\baselineskip}

\noindent{\scriptsize {\textsc{\noindent S. Banerjee and A. Budhiraja,\newline
Department of Statistics and Operations Research\newline
University of North Carolina\newline
Chapel Hill, NC 27599, USA\newline
email: sayan@email.unc.edu
\newline
email: budhiraj@email.unc.edu
\vspace{\baselineskip} } }}

\end{document}